\documentclass[10pt]{amsart}
\usepackage{times,amsmath,amsbsy,amssymb,amscd,mathrsfs}
\usepackage{graphicx,subfigure,epstopdf,wrapfig,chemarrow}

\usepackage{algorithm2e} 
\usepackage{multicol,multirow}
\usepackage{mathtools}
\usepackage[usenames,dvipsnames,svgnames,table]{xcolor}
\usepackage[all]{xy}
\usepackage{wrapfig}
\usepackage{tcolorbox}

\usepackage{tikz,tikz-cd}
\usepackage[utf8]{inputenc}
\usepackage{pgfplots} 
\usepackage{pgfgantt}
\usepackage{pdflscape}
\pgfplotsset{compat=newest} 
\pgfplotsset{plot coordinates/math parser=false}
\newlength\fwidth

\definecolor{myBlue}{rgb}{0.0,0.0,0.55}
\usepackage[colorlinks=true,citecolor=myBlue,linkcolor=myBlue]{hyperref}

\usepackage[hyperpageref]{backref}

\usepackage{comment,enumerate,multicol,xspace}

  \newcounter{mnote}
  \let\oldmarginpar\marginpar
    \renewcommand\marginpar[1]{\-\oldmarginpar[\raggedleft\footnotesize #1]%
    {\raggedright\footnotesize #1}}

\newtheorem{theorem}{Theorem}[section]
\newtheorem{lemma}[theorem]{Lemma}

\newtheorem{remark}[theorem]{Remark}

\newcommand{\dx}{\,{\rm d}x}

\newcommand{\bs}{\boldsymbol}

\renewcommand{\div}{\operatorname{div}}
\newcommand{\tr}{\operatorname{tr}}

\newcommand{\vertiii}[1]{{\left\vert\kern-0.25ex\left\vert\kern-0.25ex\left\vert #1 
    \right\vert\kern-0.25ex\right\vert\kern-0.25ex\right\vert}}

\usepackage{booktabs}
\usepackage{stmaryrd}
\usepackage{scalefnt}
\usepackage{graphicx}
\usepackage{caption}
\usepackage{float}

\usepackage{booktabs}
\usepackage{multirow}
\usepackage{siunitx}
\usepackage{tikz}
\usetikzlibrary{arrows.meta,positioning}

\allowdisplaybreaks[1]

\begin{document}
 \title[Symmetric-Tensor Distributional Mixed Method]{Symmetric-Tensor Distributional Mixed Method for Fourth-Order Elliptic Singular Perturbation Problem}
 \author{Xuehai Huang}%
 \address{School of Mathematics, Shanghai University of Finance and Economics, Shanghai 200433, China}%
 \email{huang.xuehai@sufe.edu.cn}%
 \author{Xinyue Zhao}%
 \address{School of Mathematics, Shanghai University of Finance and Economics, Shanghai 200433, China}%
 \email{zhaoxinyue20210921@163.com}%

 \makeatletter
 \@namedef{subjclassname@2020}{\textup{2020} Mathematics Subject Classification}
 \makeatother

\begin{abstract}
A symmetric-tensor distributional mixed method for a fourth-order elliptic singular perturbation problem is developed in this paper.
The moment variable is approximated by normal--normal continuous symmetric tensor elements, while the scalar variable is represented by an \(H^1\)-nonconforming virtual element space coupled with a polynomial multiplier on interior codimension-two subsimplices.
Optimal parameter-uniform error estimates are derived, independent of the presence of boundary layers.
A hybridized form is further shown to be equivalent to stabilization-free weak Galerkin and \(H^2\)-nonconforming virtual element formulations.
In two dimensions, we establish a close connection between the distributional mixed method and the classical
Hellan--Herrmann--Johnson (HHJ) method by identifying the scalar virtual element--multiplier pair with the Lagrange finite element
space.
Consequently, the proposed method extends the two-dimensional HHJ framework to any spatial dimension \(d\ge2\). Three-dimensional numerical experiments support the theoretical convergence and robustness estimates.
A two-dimensional adaptive constant-load benchmark on an L-shaped polygonal
domain tests the method on a non-manufactured nonsmooth problem and shows mesh
concentration near the reentrant corner and, for small \(\varepsilon\), boundary
refinement at the expected \(O(\varepsilon)\) scale.
\end{abstract}

\maketitle

\section{Introduction}

We develop a symmetric-tensor distributional mixed method for the fourth-order elliptic singular perturbation problem:
\begin{equation}\label{eq:fourthorderequation}
	\begin{cases}
		\varepsilon^2\Delta^2 u-\Delta u=f & \text{in }\Omega,\\
		u=\partial_n u=0 & \text{on }\partial\Omega,
	\end{cases}
\end{equation}
where \(\Omega\subset\mathbb R^d\) (\(d\ge2\)) is a bounded polytope and \(\varepsilon>0\) is the perturbation parameter. As \(\varepsilon\to0\), the operator formally
degenerates to the Poisson operator, whereas the normal-derivative condition
\(\partial_n u=0\) has no counterpart in the reduced problem. This
incompatibility may generate boundary layers near \(\partial\Omega\) and is the
main difficulty in deriving optimal error estimates uniform in \(\varepsilon\).

Several scalar primal discretizations have been proposed for
\eqref{eq:fourthorderequation}. Conforming \(H^2\) methods discretize the fourth- and second-order terms directly, but require \(C^1\) finite
elements and are difficult to implement in general \cite{Semper1992}.
Many \(C^1\)-free primal methods are robust in suitable norms, including nonconforming \(H^2\) methods \cite{NilssenTaiWinther2001}, \(C^0\) interior-penalty methods \cite{BrennerNeilan2011}, Morley--Wang--Xu type methods \cite{WangXuHu2006,WangMeng2007,WangHuangTangZhou2018,HuangShiWang2021}, and virtual element methods \cite{ZhangZhaoChen2020,FengYu2024,ZhangZhao2024}. However, for many such methods, the available parameter-uniform estimates are suboptimal in the boundary-layer regime.

Optimal or high-order
parameter-uniform results within scalar primal formulations have been obtained
by restoring compatibility with the reduced Poisson problem or by using
layer-adapted meshes, including Nitsche-type methods
\cite{Nitsche1971,GuzmanLeykekhmanNeilan2012}, hybrid high-order methods
\cite{DongErn2021}, Nitsche-modified Morley--Wang--Xu methods
\cite{HuangShiWang2021}, and layer-adapted \(C^0\) interior-penalty methods
\cite{FranzRoosWachtel2014}.
These approaches typically involve penalty or stabilization terms. The interpolation-based method of \cite{CuiHuang2025} also gives optimal parameter-uniform estimates by introducing an interpolation from an \(H^2\) finite element space to an \(H^1\) finite element space in both the load term and the Laplacian bilinear form.

Mixed formulations provide another approach by exposing the fourth-order structure through a moment tensor. For \eqref{eq:fourthorderequation}, the Hellan--Herrmann--Johnson (HHJ)
method \cite{Hellan1967,Herrmann1967,Johnson1973} was used in two dimensions in \cite{LiuHuangWang2020}, but the parameter-robust estimates proved below were not established there.
Robust and optimal mixed methods based on row-wise \(H(\operatorname{div})\)-conforming tensor discretizations were developed in \cite{HuangTang2025}; in that setting the tensor variable is not symmetric.
In the present problem, however, the moment variable is naturally symmetric because it is proportional to the Hessian. In Kirchhoff--Love plate models \cite{Ciarlet1997Plates}, this variable may be interpreted as a scaled bending-moment tensor. The symmetry of such stress or moment tensors is consistent with the symmetry of the Cauchy stress tensor in continuum mechanics, which follows from angular-momentum balance; see, e.g., \cite{GurtinFriedAnand2010}.

The central idea of this paper is to introduce
\(\boldsymbol{\sigma}=\varepsilon^2\nabla^2u\) as a symmetric tensor field. Then \eqref{eq:fourthorderequation} can be written as
\begin{equation}\label{eq:intro-mixed-system}
	\begin{cases}
		\varepsilon^{-2}\boldsymbol{\sigma}=\nabla^2u & \text{in }\Omega,\\
		\operatorname{div}\operatorname{div}\boldsymbol{\sigma}-\Delta u=f
		& \text{in }\Omega,\\
		u=\partial_nu=0 & \text{on }\partial\Omega.
	\end{cases}
\end{equation}
The resulting distributional mixed formulation \eqref{eq:mixform} seeks
\(\boldsymbol{\sigma}\in H^{-1}(\operatorname{div}\operatorname{div},\Omega;\mathbb S)\) and \(u\in H_0^1(\Omega)\), where
\begin{equation*}
 H^{-1}(\div\div,\Omega;\mathbb S)
 :=\{\boldsymbol\tau\in L^2(\Omega;\mathbb S):
 \div\div\boldsymbol\tau\in H^{-1}(\Omega)\},
\end{equation*}
and \(\mathbb S\) denotes the space of symmetric tensors. The well-posedness of this formulation is tied to the exactness of the terminal part of the distributional $\div\div$ complex
\begin{equation}\label{intro:shortdistritdivdiv}
	H^{-1}(\div\div,\Omega;\mathbb S)\xrightarrow{\div\div}H^{-1}(\Omega)\to0.
\end{equation}

At the discrete level, we use the distributional finite element
\(\operatorname{div}\operatorname{div}\) complexes in
\cite[Theorem~5.11]{ChenHuang2025}. The moment variable is approximated by
normal--normal continuous symmetric tensor elements
\cite{ChenHuang2025,HuLin2025,BerchenkoKoganGawlik2025}. The scalar
approximation uses an \(H^1\)-nonconforming virtual element space and a
polynomial multiplier on interior codimension-two subsimplices. The resulting
distributional mixed method \eqref{eq:fourth-order-discrete-mixed} possesses
optimal parameter-uniform error estimates
\eqref{eq:robusterror1}--\eqref{eq:robusterror2}, independent of the presence
of boundary layers.

We also provide an alternative construction of the normal--normal continuous
symmetric tensor element based on a tangential--normal decomposition. In this
construction, all boundary degrees of freedom are located on codimension-one
faces. Such normal--normal continuous symmetric tensor elements also arise in
the tangential-displacement
and normal--normal-stress (TDNNS) method for elasticity
\cite{PechsteinSchoeberl2011} and Reissner--Mindlin plates
\cite{PechsteinSchoeberl2017}, as well as low-order mixed methods for
elasticity \cite{CarstensenHeuer2025,CarstensenHeuer2026}.

In two dimensions, we establish a precise connection between the distributional mixed method \eqref{eq:fourth-order-discrete-mixed} and the classical HHJ method. Specifically, the \(H^1\)-nonconforming virtual element space together with the codimension-two multiplier can be identified with a Lagrange finite element space; see Fig.~\ref{fig:hessian-coupling}. Under this identification, the discrete distributional bilinear form of the biharmonic operator agrees with the HHJ bilinear form. The proposed method can therefore be viewed as an extension of the two-dimensional HHJ method to any spatial dimension \(d\ge2\). The standard HHJ coupling with a Lagrange scalar space is essentially two-dimensional and does not directly yield a discretization in higher dimensions. We overcome this obstruction by coupling the moment space with the \(H^1\)-nonconforming virtual element space and the polynomial multiplier.

To complement the manufactured convergence tests, we include a non-manufactured
constant-load clamped-plate benchmark on an L-shaped polygonal domain. The test
compares adaptive meshes for \(\varepsilon=1\) and \(\varepsilon=10^{-3}\) and
visualizes the recovered curvature field. The comparison shows mesh
concentration near the reentrant corner and, for \(\varepsilon=10^{-3}\),
boundary refinement at the expected \(O(\varepsilon)\) scale.

\begin{figure}[t]
	\centering
	\includegraphics[width=0.82\textwidth]{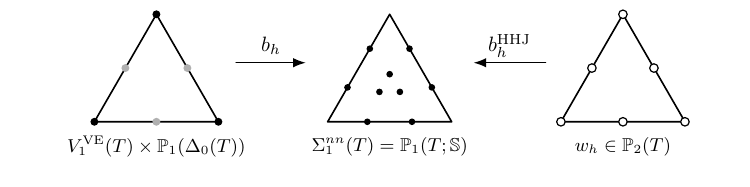}
	\caption{Lowest-order two-dimensional relation with the classical HHJ method.
		Under the local isomorphism, the pair
		\(V_1^{\mathrm{VE}}(T)\times\mathbb P_1(\Delta_0(T))\) corresponds to the
		local quadratic Lagrange finite element space, and the local bilinear form
		\(b_h\) coincides with the classical HHJ bilinear form
		\(b_h^{\mathrm{HHJ}}\).}
	\label{fig:hessian-coupling}
\end{figure}

As a further consequence, relaxing the normal--normal continuity gives a
hybridized formulation. After local elimination of the tensor variable, the
hybridized method is equivalent to stabilization-free weak Galerkin and
\(H^2\)-nonconforming virtual element methods.

The rest of the paper is organized as follows. Section~\ref{sec:distribmixedform} presents the continuous
distributional mixed formulation. Section~\ref{sec:discrete-spaces} constructs the discrete spaces and establishes the
norm equivalence. Section~\ref{sec:distribmfem} introduces the mixed method, establishes its
two-dimensional HHJ connection, and proves the robust error estimates. Section~\ref{sec:hybridization} derives the hybridized formulation and its equivalent
weak Galerkin and \(H^2\)-nonconforming virtual element formulations.
Section~\ref{sec:numerresults} reports numerical experiments.
 
  \section{Symmetric-Tensor Distributional Mixed Formulation}\label{sec:distribmixedform}
In this section, we present a symmetric-tensor distributional mixed formulation of
 \eqref{eq:fourthorderequation}, establish its well-posedness and equivalence with
 the primal weak formulation, and state the regularity assumptions used
 in the error analysis.
 
 \subsection{Notation}
 
 Throughout the paper, \(\Omega\subset\mathbb R^d\) (\(d\ge2\)) denotes a bounded
 polytope. For a bounded domain \(D\subset\mathbb R^d\) and \(m\ge0\), we denote
 by \(H^m(D)\) the standard Sobolev space on \(D\) with norm
 \(\|\cdot\|_{m,D}\) and seminorm \(|\cdot|_{m,D}\). We write \(H_0^1(D)\) and
 \(H_0^2(D)\) for the closures of \(C_0^\infty(D)\) in \(H^1(D)\) and
 \(H^2(D)\), respectively, and use \((\cdot,\cdot)_D\) to denote the
 \(L^2(D)\) inner product. When \(D=\Omega\), the subscript is omitted. The
 duality pairing between a Banach space \(V\) and its dual \(V'\) is denoted by
 \(\langle\cdot,\cdot\rangle_{V'\times V}\), or simply by
 \(\langle\cdot,\cdot\rangle\) if no confusion arises.
 
For an integer \(k\ge 0\), let \(\mathbb P_k(D)\) denote the space of
polynomials on \(D\) of total degree at most \(k\). If \(D\) is a vertex,
we identify \(\mathbb P_k(D)\) with \(\mathbb R\). We adopt the convention
\(\mathbb P_k(D)=\{0\}\) for \(k<0\). Let
\[
\mathbb S:=\{\boldsymbol\tau\in\mathbb R^{d\times d}:
\boldsymbol\tau^{\intercal}=\boldsymbol\tau\}
\]
be the space of symmetric tensors. For
\(\boldsymbol\tau\in\mathbb R^{d\times d}\), set
\[
\operatorname{sym}\boldsymbol\tau
:=\frac12(\boldsymbol\tau+\boldsymbol\tau^{\intercal}).
\]
For a finite-dimensional space \(\mathbb X\), we use
\(L^2(D;\mathbb X)\) and \(H^m(D;\mathbb X)\) for the corresponding
\(\mathbb X\)-valued Lebesgue and Sobolev spaces, and set
\(\mathbb P_k(D;\mathbb X):=\mathbb P_k(D)\otimes\mathbb X\).
We denote by \(Q_{k,D}\) the \(L^2(D;\mathbb X)\)-orthogonal projection
onto \(\mathbb P_k(D;\mathbb X)\), with the convention \(Q_{-1,D}=0\).

 For a scalar function \(v\), we use \(\nabla v\) and \(\nabla^2 v\) for its
 gradient and Hessian, respectively. For a vector-valued function
 \(\boldsymbol v\), we write \(\nabla\boldsymbol v:=\nabla\otimes\boldsymbol v\).
 For a tensor-valued function \(\boldsymbol\tau=(\tau_{ij})_{i,j=1}^d\), the
 divergence is taken row-wise, namely,
 \((\div\boldsymbol\tau)_i:=\sum_{j=1}^d\partial_j\tau_{ij}\) for \(i=1,\dots,d\),
 and \(\div\div\boldsymbol\tau\) is understood in the distributional sense
 whenever needed.
 
 Let \(\mathcal T_h\) be a conforming shape-regular simplicial triangulation of
 \(\Omega\). For each \(T\in\mathcal T_h\), we write
 \(h_T:=\operatorname{diam}(T)\) and set
 \(h:=\max_{T\in\mathcal T_h}h_T\). We denote by \(\Delta_j(T)\) the set of all
 \(j\)-dimensional subsimplices of \(T\). The sets of all faces and all interior
 faces are denoted by \(\mathcal F_h\) and \(\mathring{\mathcal F}_h\),
 respectively, while \(\mathcal E_h\) and \(\mathring{\mathcal E}_h\) denote the
 sets of all \((d-2)\)-dimensional subsimplices and the interior ones. 
For \(e\in\mathcal E_h\), we define the patch of elements around \(e\) by
 \(\omega_e:=\bigcup\{T\in\mathcal T_h:\ e\subset\overline T\}\). 
 
 For each simplex \(T\in\mathcal T_h\) with vertices \(\{\texttt{v}_i\}_{i=0}^d\), we
 denote by \(F_i\) the face opposite to \(\texttt{v}_i\), by
 \(\{\lambda_i\}_{i=0}^d\) the corresponding barycentric coordinates, and by
 \(\boldsymbol n_{\partial T}\) the unit outward normal on \(\partial T\). On
 each local face \(F_i\subset\partial T\), we write
 \(\boldsymbol n_i:=\boldsymbol n_{\partial T}|_{F_i}\). For a global
 face \(F\in\mathcal F_h\), \(\boldsymbol n_F\) denotes a fixed unit normal on
 \(F\), chosen as the outward unit normal to \(\partial\Omega\) when
 \(F\subset\partial\Omega\). If \(F\in\mathring{\mathcal F}_h\), we denote by
 \(T^+\) and \(T^-\) the two elements sharing \(F\), with \(\boldsymbol n_F\)
 taken outward on \(T^+\). 
For $e\subset \partial F$, let $\boldsymbol n_{F,e}$ denote the outward unit normal to $e$ within the hyperplane containing $F$.
We also set the edge vector \(\boldsymbol t_{i,j}:=\texttt{v}_j-\texttt{v}_i\) for \(i,j=0,1,\dots,d\).
 For a face \(F\in\mathcal F_h\), \(\div_F\) denotes the surface divergence on
 \(F\). For \(\boldsymbol\tau\in\mathbb S\), its normal--normal component on a
 local face \(F\subset\partial T\) is defined by
 \(\boldsymbol\tau_{nn}:=\boldsymbol n_{\partial T}^{\intercal}
 \boldsymbol\tau\boldsymbol n_{\partial T}\).
 
For
$\mathcal G_h\in\{\mathcal T_h,\mathcal F_h,\mathring{\mathcal F}_h,
  \mathcal E_h,\mathring{\mathcal E}_h\}$,
we define the broken Sobolev space and the piecewise polynomial space by
\[
H^s(\mathcal G_h):=\prod_{G\in\mathcal G_h} H^s(G),
\qquad
\mathbb P_k(\mathcal G_h):=\prod_{G\in\mathcal G_h}\mathbb P_k(G).
\]
For a finite-dimensional space \(\mathbb X\), we set
\[
H^s(\mathcal G_h;\mathbb X):=H^s(\mathcal G_h)\otimes \mathbb X,
\qquad
\mathbb P_k(\mathcal G_h;\mathbb X):=
\mathbb P_k(\mathcal G_h)\otimes \mathbb X .
\]
We denote by \(Q_{k,\mathcal G_h}\) the elementwise
\(L^2\)-orthogonal projection onto
\(\mathbb P_k(\mathcal G_h;\mathbb X)\), and set
\(Q_{k,h}:=Q_{k,\mathcal T_h}\). Here and throughout, functions defined
on interior faces or interior \((d-2)\)-subsimplices are understood to be
extended by zero to the corresponding boundary entities whenever needed.

 For piecewise smooth scalar-, vector-, or tensor-valued functions, we use
 \(\nabla_h\), \(\nabla_h^2\), \(\div_h\), and \(\div\div_h\) to denote the
 elementwise gradient, Hessian, divergence, and double divergence, respectively.
 We also use the broken norm and seminorm
 \[
 \|v\|_{m,h}^2:=\sum_{T\in\mathcal T_h}\|v\|_{m,T}^2,
 \qquad
 |v|_{m,h}^2:=\sum_{T\in\mathcal T_h}|v|_{m,T}^2,
 \]
 with the same notation understood componentwise for vector- and tensor-valued
 functions.
 
 For a face \(F\in\mathcal F_h\) and a piecewise smooth scalar-, vector-, or
 tensor-valued function \(\phi\), we define the jump by
 \[
 [\![\phi]\!]|_F :=
 \begin{cases}
 \phi^+\boldsymbol{n}_F\cdot\boldsymbol{n}_{\partial T^+} + \phi^-\boldsymbol{n}_F\cdot\boldsymbol{n}_{\partial T^-},
 & \text{if } F=\partial T^+\cap\partial T^-\in\mathring{\mathcal F}_h,\\
 \phi|_F,
 & \text{if } F\subset\partial\Omega,
 \end{cases}
 \]
 where \(\phi^\pm\) denote the traces from \(T^\pm\).
 
Throughout this paper, \(C\) denotes a generic positive constant independent of the mesh size
 \(h\) and the perturbation parameter \(\varepsilon\). We write \(a\lesssim b\) if
 \(a\le Cb\), and \(a\eqsim b\) if \(a\lesssim b\) and \(b\lesssim a\).

 We shall use the following Green identity for the \(\operatorname{div}\operatorname{div}\) operator; see \cite[Lemma~5.2]{ChenHuang2022} and \cite[Lemma~4.1]{ChenHuang2022a}.
 \begin{lemma}
For any $\boldsymbol \sigma\in \mathcal C^2(T; \mathbb S)$ and $v\in H^2(T)$, one has
\begin{equation}\label{eq:greenidentitydivdiv}
\begin{aligned}
(\div\div\boldsymbol \sigma, v)_T&=(\boldsymbol \sigma, \nabla^2v)_T - ( \tr_1(\bs \sigma), \partial_nv)_{\partial T} +  ( \tr_2(\bs \sigma), v)_{\partial T} \\
&\quad -\sum_{e\in\Delta_{d-2}(T)}(\tr_e(\bs \sigma), v)_e, 
\end{aligned}
\end{equation}
where
\begin{equation}\label{eq:divdiv-traces}
\begin{aligned}
&\tr_1(\bs \sigma) = \boldsymbol  n_{\partial T}^{\intercal}\boldsymbol \sigma\boldsymbol  n_{\partial T}, \quad \tr_2(\bs \sigma) =  \boldsymbol n_{\partial T}^{\intercal}\div \boldsymbol \sigma +  \div_F(\boldsymbol\sigma \boldsymbol n_{\partial T}), \\
&\tr_e(\bs \sigma) = \sum_{F\in\partial T,e\in \partial F}\boldsymbol n_{F,e}^{\intercal}\boldsymbol \sigma \boldsymbol n_{\partial T}.
\end{aligned}
\end{equation}
\end{lemma}

 \subsection{A distributional mixed formulation}
 
 We formulate \eqref{eq:intro-mixed-system} in a symmetric-tensor distributional
 setting. The symmetric tensor variable \(\boldsymbol{\sigma}\) is kept as a
 primary unknown, with only the natural regularity dictated by the second
 equation in \eqref{eq:intro-mixed-system}. Since
 \(\div\div\boldsymbol{\sigma}\) is tested against functions in \(H_0^1(\Omega)\),
 it is natural to introduce
 \[
 H^{-1}(\div\div,\Omega;\mathbb S)
 :=
 \left\{
 \boldsymbol\tau\in L^2(\Omega;\mathbb S):
 \div\div\boldsymbol\tau\in H^{-1}(\Omega)
 \right\},
 \]
 equipped with the norm
 \[
 \|\boldsymbol\tau\|_{H^{-1}(\div\div)}^2
 :=
 \|\boldsymbol\tau\|_0^2+\|\div\div\boldsymbol\tau\|_{-1}^2,\;
 \|\div\div\boldsymbol\tau\|_{-1}
 :=
 \sup_{v\in H_0^1(\Omega),\,v\neq0}
 \frac{\langle \div\div\boldsymbol\tau,v\rangle}{|v|_1}.
 \]
 For the singularly perturbed problem, we also use the parameter-dependent norm
 \[
 \|\boldsymbol\tau\|_{\varepsilon^{-1}L^2\cap H^{-1}(\div\div)}^2
 :=
 \varepsilon^{-2}\|\boldsymbol\tau\|_0^2
 +\|\div\div\boldsymbol\tau\|_{-1}^2.
 \]
 The distributional mixed formulation is to find
 \((\boldsymbol{\sigma},u)\in H^{-1}(\div\div,\Omega;\mathbb S)\times H_0^1(\Omega)\)
 such that
 \begin{subequations}\label{eq:mixform}
 \begin{align}
 a(\boldsymbol{\sigma}, \boldsymbol{\tau})
 + b(\boldsymbol{\tau}, u)
 &= 0,
 && \forall\, \boldsymbol{\tau} \in
 H^{-1}(\div\div,\Omega;\mathbb S),
 \label{eq:mixform1}\\
 b(\boldsymbol{\sigma}, v) - c(u, v)
 &= -(f, v),
 && \forall\, v \in H_0^1(\Omega),
 \label{eq:mixform2}
 \end{align}
 \end{subequations}
 where
 \[
 a(\boldsymbol{\sigma}, \boldsymbol{\tau})
 := \varepsilon^{-2}(\boldsymbol{\sigma}, \boldsymbol{\tau}), \qquad
 b(\boldsymbol{\tau}, v)
 := -\langle \div\div \boldsymbol{\tau}, v \rangle, \qquad
 c(u, v)
 := (\nabla u, \nabla v).
 \]
The well-posedness of the distributional mixed formulation \eqref{eq:mixform} is governed by the exactness of the terminal part of the distributional $\div\div$ complex,
\begin{equation*}
	H^{-1}(\div\div,\Omega;\mathbb S)\xrightarrow{\div\div}H^{-1}(\Omega)\to0.
\end{equation*}
 
 \begin{theorem}\label{thm:mix-well-posed}
 The mixed formulation \eqref{eq:mixform} is well posed. More precisely, it
 admits a unique solution
 \((\boldsymbol{\sigma},u)\in H^{-1}(\div\div,\Omega;\mathbb S)\times H_0^1(\Omega)\)
 and satisfies
 \[
 \|\boldsymbol{\sigma}\|_{\varepsilon^{-1}L^2\cap H^{-1}(\div\div)}
 +|u|_1
 \lesssim
 \|f\|_{-1}.
 \]
 \end{theorem}
 
 \begin{proof}
 For any \(\boldsymbol\tau\in H^{-1}(\div\div,\Omega;\mathbb S)\), the definition
 of the dual norm gives
 \[
 a(\boldsymbol\tau,\boldsymbol\tau)
 +\sup_{v\in H_0^1(\Omega),\,v\neq0}
 \frac{b(\boldsymbol\tau,v)^2}{|v|_1^2}
 =
 \|\boldsymbol\tau\|_{\varepsilon^{-1}L^2\cap H^{-1}(\div\div)}^2.
 \]
 Moreover, for all \(v\in H_0^1(\Omega)\),
 \[
 c(v,v)
 +\sup_{\boldsymbol\tau\in H^{-1}(\div\div,\Omega;\mathbb S),\,\boldsymbol\tau\neq0}
 \frac{b(\boldsymbol\tau,v)^2}
 {\|\boldsymbol\tau\|_{\varepsilon^{-1}L^2\cap H^{-1}(\div\div)}^2}
 \eqsim |v|_1^2.
 \]
Zulehner's theory \cite{Zulehner2011} therefore implies that the mixed formulation
 \eqref{eq:mixform} is well posed. The stated stability estimate follows from
 \(|(f,v)|\le \|f\|_{-1}|v|_1\) for all \(v\in H_0^1(\Omega)\).
 \end{proof}
 
 For comparison, we recall the standard primal weak formulation of
 \eqref{eq:fourthorderequation}: find \(u\in H_0^2(\Omega)\) such that
 \begin{equation}\label{eq:primal-weak}
 \varepsilon^2(\nabla^2u,\nabla^2v)+(\nabla u,\nabla v)=(f,v)
 \qquad
 \forall\,v\in H_0^2(\Omega).
 \end{equation}
 Since the bilinear form on the left-hand side of \eqref{eq:primal-weak} is
 continuous and coercive on \(H_0^2(\Omega)\), problem \eqref{eq:primal-weak}
 admits a unique solution by the Lax--Milgram theorem.
 
 \begin{lemma}\label{lem:equivalence}
 The mixed formulation \eqref{eq:mixform} is equivalent to the primal
 formulation \eqref{eq:primal-weak}.
 \end{lemma}
 
 \begin{proof}
 By the uniqueness of the mixed formulation \eqref{eq:mixform} and the primal
 formulation \eqref{eq:primal-weak}, it
 suffices to show that if \(u\in H_0^2(\Omega)\) solves
 \eqref{eq:primal-weak} and
 \(\boldsymbol{\sigma}=\varepsilon^2\nabla^2u\), then
 \((\boldsymbol{\sigma},u)\) solves \eqref{eq:mixform}.
 
 Clearly, \(\boldsymbol{\sigma}\in L^2(\Omega;\mathbb S)\). Moreover,
 \eqref{eq:primal-weak} implies
 \(\div\div\boldsymbol{\sigma}=f+\Delta u\in H^{-1}(\Omega)\), and hence
 \(\boldsymbol{\sigma}\in H^{-1}(\div\div,\Omega;\mathbb S)\). Since
 \(u\in H_0^2(\Omega)\), we have
 \[
 \langle \div\div\boldsymbol\tau,u\rangle=(\boldsymbol\tau,\nabla^2u)
 \qquad
 \forall\,\boldsymbol\tau\in H^{-1}(\div\div,\Omega;\mathbb S),
 \]
 which, together with \(\boldsymbol{\sigma}=\varepsilon^2\nabla^2u\), yields
 \eqref{eq:mixform1}.
 
 For any \(v\in C_0^\infty(\Omega)\), \eqref{eq:primal-weak} gives
 \[
 \langle \div\div\boldsymbol{\sigma},v\rangle+(\nabla u,\nabla v)=(f,v).
 \]
 By density of \(C_0^\infty(\Omega)\) in \(H_0^1(\Omega)\), this identity
 extends to all \(v\in H_0^1(\Omega)\), and thus \eqref{eq:mixform2} follows.
 \end{proof}

 \subsection{Regularity assumptions}
 To state the regularity assumptions used in the robust error analysis, we let
 \(\bar u\) denote the solution of the limiting Poisson problem
 \begin{equation}\label{eq:poisson}
 \begin{cases}
 -\Delta \bar{u} = f, & \text{in } \Omega, \\
 \bar{u} = 0, & \text{on } \partial\Omega.
 \end{cases}
 \end{equation}
We regard \(\bar u\) as the regular part of the solution, and
 \(u-\bar u\) as the singularly perturbed remainder.
 
 We assume the elliptic regularity estimate: for some \(s\ge2\),
 \begin{equation}\label{eq:regularity-assumption}
 \|\bar{u}\|_{s} \lesssim \|f\|_{s-2}.
 \end{equation}
 If \(\Omega\) is semiconvex, or if the closure of \(\Omega\) has uniformly
 positive reach, estimate \eqref{eq:regularity-assumption} for \(s=2\) holds;
 see
 \cite{Kadlec1964,Talenti1965,Adolfsson1992,MitreaMitreaYan2010,GaoLai2020}.
 In particular, every convex domain is semiconvex. 
 
 We further assume that
 \begin{equation}\label{eq:perturbation-regularity}
 |u-\bar{u}|_{1}
 + \varepsilon |u|_{2}
 + \varepsilon^{2} |u|_{3}
 \lesssim \varepsilon^{1/2}\|f\|_{0}.
 \end{equation}
 Estimate \eqref{eq:perturbation-regularity} is available on convex domains in
 two and three dimensions; see, for example,
 \cite{GuzmanLeykekhmanNeilan2012,NilssenTaiWinther2001}. More recently, \cite[Theorem~2]{LiMingZhou2025TRUNC} established a related
 arbitrary-dimensional counterpart of \eqref{eq:perturbation-regularity}.
 
 
 \section{Discrete Spaces, Weak Hessian and Norm Equivalences}\label{sec:discrete-spaces}
 
 This section introduces the discrete spaces and interpolation operators used in the mixed method, and then proves the weak-Hessian norm equivalence on which the stability analysis rests.

 \subsection{Normal--normal continuous finite element spaces for symmetric tensors}
Finite elements for symmetric tensors with normal--normal continuity have
been constructed in
\cite{Hellan1967,Herrmann1967,Johnson1973,PechsteinSchoeberl2011,
ChenHuang2025,HuLin2025,BerchenkoKoganGawlik2025}. 
We use here a
tangential--normal construction whose boundary degrees of freedom (DoFs) are supported only on faces.

For each simplex \(T\in\mathcal T_h\) and each integer \(k\ge 1\), we
take \(\mathbb P_k(T;\mathbb S)\) as the local shape function space. The
DoFs are defined as follows:
\begin{subequations}\label{eq:nnDoFs}
\begin{align}
\bigl(\boldsymbol\tau_{nn},q\bigr)_F,
&\qquad q\in\mathbb P_k(F),\quad F\in\Delta_{d-1}(T),
\label{eq:nnDoFs1}\\
\bigl(\boldsymbol n_i^{\intercal}\boldsymbol\tau\,\boldsymbol t_{0,j},q\bigr)_{F_i},
&\qquad q\in\mathbb P_k(F_i),\quad
i=1,\ldots,d-2,\quad j=i+1,\ldots,d,
\label{eq:nnDoFs2}\\
(\boldsymbol\tau,\boldsymbol q)_T,
&\qquad \boldsymbol q\in\mathbb P_{k-1}(T;\mathbb S).
\label{eq:nnDoFs3}
\end{align}
\end{subequations}
For \(d=2\), the DoFs \eqref{eq:nnDoFs2} are absent.
The boundary DoFs \eqref{eq:nnDoFs1}--\eqref{eq:nnDoFs2}
are associated only with the \((d-1)\)-dimensional faces of \(T\). The
interior moments in \eqref{eq:nnDoFs3} are taken against the full
symmetric tensor-valued polynomial space of degree at most \(k-1\).

To prove unisolvence of the DoFs \eqref{eq:nnDoFs} for \(\mathbb P_k(T;\mathbb S)\), we use the following direct-sum decomposition of \(\mathbb S\), adapted to the tangential--normal decomposition of \cite{Chen2024,ChenChenHuangWei2024}.

 \begin{lemma}\label{lem:basis-S}
 The space \(\mathbb S\) admits the direct sum decomposition
 \begin{equation}\label{eq:S-decomposition}
 \mathbb S
 =
 \bigoplus_{i=0}^{d} \operatorname{span}\{\boldsymbol n_i \boldsymbol n_i^{\intercal}\}
 \oplus
 \bigoplus_{i=1}^{d-2}
 \bigoplus_{j=i+1}^{d}
 \operatorname{span}\{\operatorname{sym}(\boldsymbol n_i \boldsymbol t_{0,j}^{\intercal})\}.
 \end{equation}
 \end{lemma}
 
 \begin{proof}
 The number of tensors on the right-hand side of
 \eqref{eq:S-decomposition} is
 \[
 (d+1)+\sum_{i=1}^{d-2}(d-i)=\frac12 d(d+1)=\dim\mathbb S.
 \]
 It therefore suffices to prove linear independence. Suppose that
 \begin{equation}\label{eq:lin-indep}
 \sum_{i=0}^{d} C_{ii}\,\boldsymbol n_i\boldsymbol n_i^{\intercal}
 +
 \sum_{i=1}^{d-2}\sum_{j=i+1}^{d}
 C_{ij}\,\operatorname{sym}(\boldsymbol n_i\boldsymbol t_{0,j}^{\intercal})
 =\boldsymbol 0,
 \end{equation}
 with \(C_{ii},C_{ij}\in\mathbb R\). We use
 \[
 \boldsymbol n_i^{\intercal}\boldsymbol t_{0,j}=0
 \quad\text{for }1\le i\ne j\le d;
 \qquad
 \boldsymbol n_0^{\intercal}\boldsymbol t_{0,j}\ne0,\quad
 \boldsymbol n_j^{\intercal}\boldsymbol t_{0,j}\ne0
 \quad\text{for }j=1,\dots,d.
 \]
 Multiplying \eqref{eq:lin-indep} on the right by \(\boldsymbol t_{0,d}\) gives
 \[
 C_{00}\boldsymbol n_0(\boldsymbol n_0^{\intercal}\boldsymbol t_{0,d})
 +
 C_{dd}\boldsymbol n_d(\boldsymbol n_d^{\intercal}\boldsymbol t_{0,d})
 +
 \frac12\sum_{i=1}^{d-2}\sum_{j=i+1}^{d}
 C_{ij}\boldsymbol n_i(\boldsymbol t_{0,j}^{\intercal}\boldsymbol t_{0,d})
 =\boldsymbol 0.
 \]
 Since
 \(\{\boldsymbol n_0,\boldsymbol n_1,\ldots,\boldsymbol n_{d-2},\boldsymbol n_d\}\)
 is a basis of \(\mathbb R^d\), we obtain \(C_{00}=C_{dd}=0\). Similarly,
 multiplication by \(\boldsymbol t_{0,d-1}\) yields \(C_{d-1,d-1}=0\).
 
Now proceed by backward induction on $m=d-2, \ldots, 1$. Assume that
 \(C_{ij}=0\) for all \(i=m+1,\dots,d-2\) and \(j=i,\dots,d\). 
Then equation \eqref{eq:lin-indep} reduces to
 \begin{equation}\label{eq:lin-indepm}
 \sum_{i=1}^{m} C_{ii}\,\boldsymbol n_i\boldsymbol n_i^{\intercal}
 +
 \sum_{i=1}^{m}\sum_{j=i+1}^{d}
 C_{ij}\,\operatorname{sym}(\boldsymbol n_i\boldsymbol t_{0,j}^{\intercal})
 =\boldsymbol 0.
 \end{equation}
Multiplying
 \eqref{eq:lin-indepm} on the right by \(\boldsymbol t_{0,\ell}\) for
 \(\ell=m+1,\dots,d\), and then extracting the coefficient of \(\boldsymbol n_m\), we obtain
 \[
 \sum_{j=m+1}^{d}(\boldsymbol t_{0,\ell}\cdot \boldsymbol t_{0,j})C_{mj}
 =0,
 \qquad \ell=m+1,\dots,d.
 \]
 Since \(\{\boldsymbol t_{0,m+1},\dots,\boldsymbol t_{0,d}\}\) is linearly
 independent, its Gram matrix is nonsingular. Hence
 \[
 C_{m,m+1}=C_{m,m+2}=\cdots=C_{m,d}=0.
 \]
 Multiplying \eqref{eq:lin-indepm} on the right by \(\boldsymbol t_{0,m}\)
 then gives \(C_{mm}=0\).
 
 Hence all coefficients in \eqref{eq:lin-indep} vanish, proving linear
 independence.
 \end{proof}
 
 Let \(\{N_m\}_{m=1}^{\dim\mathbb S}\) be the basis of \(\mathbb S\) given by
 \eqref{eq:S-decomposition}, and let \(\{N_m'\}_{m=1}^{\dim\mathbb S}\) be the
 dual basis, defined by \(N_m':N_\ell=\delta_{m\ell}\). 
Here $\delta_{m\ell}$ denotes the Kronecker delta.

 \begin{lemma}\label{lem:nn-unisolvence}
 For \(k\ge1\), the DoFs \eqref{eq:nnDoFs} are unisolvent for
 \(\mathbb P_k(T;\mathbb S)\).
 \end{lemma}
 
 \begin{proof}
 Using
 \[
 (d+1)+\sum_{i=1}^{d-2}(d-i)=\dim\mathbb S
 \quad\text{and}\quad
 \dim\mathbb P_k(T)=\dim\mathbb P_{k-1}(T)+\dim\mathbb P_k(F),
 \]
 we see that the number of DoFs in \eqref{eq:nnDoFs} is
 \(\dim\mathbb P_k(T;\mathbb S)\).
 
 Let \(\boldsymbol\tau\in\mathbb P_k(T;\mathbb S)\) and assume that all the DoFs in
 \eqref{eq:nnDoFs1}--\eqref{eq:nnDoFs3} vanish.
  Expanding \(\boldsymbol\tau\) with respect to the dual basis, we write
 \[
 \boldsymbol\tau=\sum_{m=1}^{\dim\mathbb S} c_m\,N_m',
 \qquad
 c_m:=\boldsymbol\tau:N_m\in\mathbb P_k(T).
 \]
 For each \(m\), let \(F_{j_m}\in\Delta_{d-1}(T)\) be the face associated with
 \(N_m\), namely \(j_m=i\) whenever
 \(N_m=\boldsymbol n_i\boldsymbol n_i^{\intercal}\) or
 \(N_m=\operatorname{sym}(\boldsymbol n_i\boldsymbol t_{0,j}^{\intercal})\).
 In the former case,
 \(c_m=\boldsymbol n_i^{\intercal}\boldsymbol\tau\,\boldsymbol n_i\), while in
 the latter,
 \(c_m=\boldsymbol n_i^{\intercal}\boldsymbol\tau\,\boldsymbol t_{0,j}\). Since
 the corresponding face moments \eqref{eq:nnDoFs1}--\eqref{eq:nnDoFs2} vanish and
 \(c_m|_{F_{j_m}}\in\mathbb P_k(F_{j_m})\), we obtain
 \(c_m|_{F_{j_m}}=0\). Hence
 \(c_m=\lambda_{j_m}q_m\) for some \(q_m\in\mathbb P_{k-1}(T)\), and therefore
 \[
 \boldsymbol\tau=\sum_{m=1}^{\dim\mathbb S}\lambda_{j_m}q_m\,N_m'.
 \]
 
 Fixing \(m\) and testing \eqref{eq:nnDoFs3} with
 \(\boldsymbol q=q_mN_m\in\mathbb P_{k-1}(T;\mathbb S)\), we obtain
 \[
 0=(\boldsymbol\tau,q_mN_m)_T
 =
 \sum_{\ell=1}^{\dim\mathbb S}
 (\lambda_{j_\ell}q_\ell N_\ell',q_mN_m)_T
 =
 (\lambda_{j_m}q_m,q_m)_T
 =
 \int_T \lambda_{j_m}q_m^2\dx,
 \]
 where we used \(N_\ell':N_m=\delta_{\ell m}\). Since \(\lambda_{j_m}>0\) in \(T\),
 we obtain \(q_m=0\), and hence \(\boldsymbol\tau=0\).
 \end{proof}
 
 Accordingly, we define the global finite element space
 \[
 \Sigma_{k,h}^{nn}
 :=
\{
 \boldsymbol\tau\in \Sigma^{nn}(\mathcal T_h;\mathbb S):
 \boldsymbol\tau|_T\in\mathbb P_k(T;\mathbb S)\;\; \forall\,T\in\mathcal T_h
\},
 \]
where
 \[
 \Sigma^{nn}(\mathcal T_h;\mathbb S)
 :=
\{
 \boldsymbol\tau\in H^1(\mathcal T_h;\mathbb S)
 :
 [\![\boldsymbol\tau_{nn}]\!]|_F=0
 \;\;
 \forall\, F\in\mathring{\mathcal F}_h
\}.
 \]
 Thus \(\Sigma_{k,h}^{nn}\) imposes only normal--normal continuity across
 interelement faces, with no additional edge or vertex continuity constraints.

 For \(\boldsymbol\tau\in\Sigma^{nn}(\mathcal T_h;\mathbb S)\), let
 \(I_h^{nn}\boldsymbol\tau\in\Sigma_{k,h}^{nn}\) be the canonical interpolant
 associated with the DoFs \eqref{eq:nnDoFs}. By a standard scaling
 argument, for any \(1\le s\le k+1\) and any \(T\in\mathcal T_h\), we have
 \begin{equation}\label{eq:intererrorSigma}
 \|\boldsymbol\tau-I_h^{nn}\boldsymbol\tau\|_{0,T}
 \lesssim
 h_T^{s}\,|\boldsymbol\tau|_{s,T},
 \qquad
 \forall\,\boldsymbol\tau\in
 \Sigma^{nn}(\mathcal T_h;\mathbb S)
 \cap H^{s}(\mathcal T_h;\mathbb S), T\in\mathcal T_h.
 \end{equation}
 
 \subsection{$H^1$-nonconforming virtual element space}
 
 For the scalar variable, we use the \(H^1\)-nonconforming virtual element space
 introduced in \cite{De2016nonconforming,ChenHuang2020nonconforming}. This space
 is compatible with the weak continuity structure of \(\Sigma_{k,h}^{nn}\) and
 the weak Hessian framework developed below. For each integer \(k\ge1\), the
 local shape function space is defined by
 \begin{equation}\label{eq:vem-space-def}
 V_{k}^{\mathrm{VE}}(T)
 :=
 \bigl\{
 v \in H^1(T):
 \Delta v \in \mathbb P_{k-2}(T),
 \partial_n v|_F \in \mathbb P_{k-1}(F)
 \; \forall\, F \in \Delta_{d-1}(T)
 \bigr\}.
 \end{equation}
 The space \(V_k^{\mathrm{VE}}(T)\) contains \(\mathbb P_k(T)\) and coincides
 with \(\mathbb P_1(T)\) when \(k=1\). The associated DoFs are
 \begin{subequations}\label{eq:vem-dofs}
 \begin{align}
 \frac{1}{|F|}(v,q)_F,
 &\qquad q \in \mathbb P_{k-1}(F),\;
 F \in \Delta_{d-1}(T),
 \label{eq:vem-dof-face} \\
 \frac{1}{|T|}(v,q)_T,
 &\qquad q \in \mathbb P_{k-2}(T).
 \label{eq:vem-dof-cell}
 \end{align}
 \end{subequations}
By a scaling argument, we obtain the \(L^2\)-norm equivalence
\begin{equation}\label{eq:vem-norm-equiv}
\|v\|_{0,T}^2\eqsim \|Q_{k-2,T}v\|_{0,T}^2 + \sum_{F\in\Delta_{d-1}(T)}h_F\|Q_{k-1,F}v\|_{0,F}^2 \quad\forall\,v\in V_k^{\mathrm{VE}}(T).      
\end{equation}

 The associated global \(H^1\)-nonconforming virtual element space is
 \begin{equation}\label{eq:vem-global-space}
 	\mathring V_{k,h}^{\mathrm{VE}}
 	:=
 	\{
 	v\in \mathring{H}_1^{\rm nc}(\mathcal T_h)
 	:\;
 	v|_T\in V_k^{\mathrm{VE}}(T)\;\;\forall\,T\in\mathcal T_h
 	\},
 \end{equation}
 where
 \[
 \mathring{H}_1^{\rm nc}(\mathcal T_h)
 :=
 \{
 v\in H^1(\mathcal T_h)
 :
 [\![Q_{k-1,F}v]\!]|_F=0
 \;\;
 \forall\,F\in\mathcal F_h
 \}.
 \]
 For \(k=1\), the space \(\mathring V_{1,h}^{\mathrm{VE}}\) coincides with the
 classical Crouzeix--Raviart finite element space \cite{CrouzeixRaviart1973}.
 This space satisfies the following weak continuity property:
 \begin{equation}\label{eq:vem-weak-continuity}
 ([\![v]\!], q)_F = 0
 \quad
 \forall\, v \in \mathring V_{k,h}^{\mathrm{VE}},\;
 q \in \mathbb P_{k-1}(F),\;
 F \in \mathcal F_h.
 \end{equation}
 For each \(T\in\mathcal T_h\), the following norm equivalence holds
 \cite[Lemma~3.3]{HuangTang2025}:
 \begin{equation}\label{eq:vem-grad-projection-equivalence}
 \|Q_{k-1,T}\nabla v\|_{0,T}
 \eqsim
 \|\nabla v\|_{0,T},
 \qquad
 \forall\, v\in V_k^{\mathrm{VE}}(T).
 \end{equation}
   
Let \(I_{k,h}^{\mathrm{VE}} : \mathring{H}_1^{\mathrm{nc}}(\mathcal T_h) \to
\mathring V_{k,h}^{\mathrm{VE}}\) denote the global interpolation operator
associated with the DoFs~\eqref{eq:vem-dofs}. When no ambiguity
arises, we abbreviate \(I_{k,h}^{\mathrm{VE}}\) as \(I_h^{\mathrm{VE}}\).
The following estimate holds (see \cite{ChenHuang2020nonconforming}): for any $1 \le s \le k+1$ and any $T \in \mathcal T_h$,
 \begin{equation}\label{eq:intererrorVE}
 \|v-I_h^{\mathrm{VE}}v\|_{0,T}
 +
 h_T|v-I_h^{\mathrm{VE}}v|_{1,T}
 \lesssim
 h_T^{s} |v|_{s,T},
 \qquad
 \forall\, v\in H_0^1(\Omega)\cap H^{s}(\mathcal T_h).
 \end{equation}
For $v\in\mathring{H}_1^{\mathrm{nc}}(\mathcal T_h)$, let \(v^{\mathrm{CR}}:=I_{1,h}^{\mathrm{VE}}v\in \mathring V_{1,h}^{\mathrm{VE}}\) be the
nonconforming linear interpolant of \(v\).

We also need the following estimate relating face
 jumps to the multiplier.
 \begin{lemma}\label{lem:CRjump-by-edge}
Let $F \in \mathcal{F}_h$. It holds for \(v \in \mathring{V}_{k,h}^{\mathrm{VE}}\) and
 \(\mu \in\mathbb{P}_k(\mathring{\mathcal{E}}_h)\) that
 \begin{equation}\label{eq:CRjump-by-edge-local}
 h_F^{-1} \,\|[\![v^{\mathrm{CR}}]\!]\|_{0,F}^2
 \lesssim \sum_{e \in \Delta_{d-2}(F)} \|v^{\mathrm{CR}}-\mu\|_{0,e}^2.
 \end{equation}
 \end{lemma}
 \begin{proof}
By $[\![v^{\mathrm{CR}}]\!]|_F\in\mathbb P_1(F)$, the DoFs~\eqref{eq:vem-dofs}
for the nonconforming linear element and the norm equivalence
\eqref{eq:vem-norm-equiv}, we have for each face \(F\),
\[
h_F^{-1}\|[\![v^{\mathrm{CR}}]\!]\|_{0,F}^2 \eqsim \sum_{e \in \Delta_{d-2}(F)}
 \|Q_{0,e}[\![v^{\mathrm{CR}}]\!]\|_{0,e}^2\leq \sum_{e \in \Delta_{d-2}(F)}
 \|[\![v^{\mathrm{CR}}]\!]\|_{0,e}^2.
\]
The estimate \eqref{eq:CRjump-by-edge-local} then follows from the fact that $\mu$ is single-valued on each interior $(d-2)$-dimensional entity and vanishes on the boundary.
 \end{proof}

By the definition of
 \(I_h^{\mathrm{VE}}\) together with \(Q_{k-1,h}\), the following commuting
 property holds \cite[Lemma~6]{HuangTang2025}:
 \begin{equation}\label{eq:vem-hdiv-commuting}
 Q_{k-1,h}\,\nabla_h(I_h^{\mathrm{VE}}v)
 =
 Q_{k-1,h}\,\nabla_h v,
 \qquad
 \forall\, v\in H_0^1(\Omega)+\mathring V_{k,h}^{\mathrm{VE}}.
 \end{equation}
 
 We also employ the standard \(H(\div)\)-conforming interpolation operator
 \(I_h^{\div}\) into the BDM space of degree \(k-1\); see, for example,
 \cite{Brezzi1987,Brezzi1985,Nedelec1986,ChenHuang2022,Chen2024}.
 For \(\boldsymbol w\in H^1(\Omega;\mathbb R^d)\), \(v_h\in\mathring V_{k,h}^{\mathrm{VE}}\),
 and \(k\ge2\), we have (cf. \cite[Lemma~7]{HuangTang2025})
 \begin{multline}\label{eq:hdiv-projection-difference}
 \bigl(
 Q_{k-1,h}\boldsymbol w-I_h^{\div}\boldsymbol w,\,
 \nabla_h v_h
 \bigr) \\
 =
 \sum_{T\in\mathcal T_h}\sum_{F\in\Delta_{d-1}(T)}
 \Bigl(
 \bigl(Q_{k-1,T}\boldsymbol w-\boldsymbol w\bigr)\cdot\boldsymbol n,\,
 Q_{k-1,F}v_h - Q_{k-2,T}v_h
 \Bigr)_F .
 \end{multline}
 In addition, \(I_h^{\div}\) satisfies the commuting relation \cite[(2.5.27)]{BoffiBrezziFortin2013}
 \begin{equation}\label{eq:commutationIhdiv}
 Q_{k-2,h}\div\boldsymbol w = \div I_h^{\div}\boldsymbol w.
 \end{equation}
 
\subsection{Weak Hessian and norm equivalence}
For \((v,\mu)\in \mathring H_1^{\mathrm{nc}}(\mathcal T_h)
\times L^2(\mathring{\mathcal E}_h)\), we define the weak Hessian
\(\nabla_w^2(v,\mu)\in \Sigma_{k,h}^{nn}\) as follows: for all \(\boldsymbol\tau\in\Sigma_{k,h}^{nn}\),
\[
 (\nabla_w^2(v,\mu),\boldsymbol{\tau})
 =
\sum_{T \in \mathcal{T}_h}\Big((v,\div\div\boldsymbol{\tau})_T
 -(v,\tr_2(\boldsymbol{\tau}))_{\partial T}
 +\sum_{e\in\Delta_{d-2}(T)} (\mu,\tr_e(\boldsymbol{\tau}))_e\Big).
 \]

\begin{lemma}\label{lem:weakHessian-properties}
The weak Hessian satisfies
\begin{equation}\label{eq:weakHessian-commute}
\nabla_w^2(I_h^{\mathrm{VE}}v,Q_{k,\mathcal E_h}\mu)
=
\nabla_w^2(v,\mu),
\qquad
\forall\,(v,\mu)\in H_0^1(\Omega)\times L^2(\mathring{\mathcal E}_h).
\end{equation}
\end{lemma}
\begin{proof}
The identity follows directly from the definitions of
\(I_h^{\mathrm{VE}}\) and \(Q_{k,\mathcal E_h}\).
\end{proof}

For later use, we introduce two discrete \(H^2\)-seminorms on
\(\mathring H_1^{\mathrm{nc}}(\mathcal T_h)\times
L^2(\mathring{\mathcal E}_h)\):
\[
\begin{aligned}
|\!|\!|(v,\mu)|\!|\!|_{2,\mathrm{CR}}^2
:=&
\sum_{T\in\mathcal T_h}
h_T^{-4}\|Q_{k-2,T}(v-v^{\mathrm{CR}})\|_{0,T}^2 \\
&
+\sum_{T\in\mathcal T_h}\sum_{F\in\Delta_{d-1}(T)}
h_T^{-3}\|Q_{k-1,F}(v-v^{\mathrm{CR}})\|_{0,F}^2  \\
&
+\sum_{T\in\mathcal T_h}\sum_{e\in\Delta_{d-2}(T)}
h_T^{-2}\|Q_{k,e}(v^{\mathrm{CR}}-\mu)\|_{0,e}^2
+\sum_{F\in\mathcal F_h}
h_F^{-1}\|[\![\partial_n v^{\mathrm{CR}}]\!]\|_{0,F}^2,
\end{aligned}
\]
and
\[
\begin{aligned}
|\!|\!|(v,\mu)|\!|\!|_{2,h}^2
:={}&
|Q_{k-1,h}\nabla_h v|_{1,h}^2
+
\sum_{F\in\mathcal F_h}
h_F^{-1}
\|[\![Q_{k-1,h}\nabla_h v]\!]\|_{0,F}^2  \\
&+
\sum_{T\in\mathcal T_h}
\sum_{e\in\Delta_{d-2}(T)}
h_T^{-2}
\|Q_{k,e}(v^{\mathrm{CR}}-\mu)\|_{0,e}^2 .
\end{aligned}
\]
By the broken Poincar\'e inequality in \cite[(1.8)]{Brenner2003} and the norm equivalence \eqref{eq:vem-grad-projection-equivalence}, the seminorm \(|\!|\!|\cdot|\!|\!|_{2,h}\) is indeed a norm on
\(\mathring V_{k,h}^{\mathrm{VE}}\times \mathbb P_k(\mathring{\mathcal E}_h)\).

\begin{lemma}\label{lem:projected-gradient-H2-equivalence}
For \(k\ge1\), the following norm equivalences hold:
\begin{equation}\label{eq:normequivalence}
\|\nabla_w^{2}(v,\mu)\|_{0}
\eqsim
|\!|\!|(v,\mu)|\!|\!|_{2,h}
\eqsim
|\!|\!|(v,\mu)|\!|\!|_{2,\mathrm{CR}},
\qquad
\forall\,(v,\mu)\in
\mathring V_{k,h}^{\mathrm{VE}}\times
\mathbb P_k(\mathring{\mathcal E}_h).
\end{equation}
\end{lemma}

\begin{proof}
Following the argument in \cite[Lemma B.1]{ChenHuang2025},
\[
\|\nabla_w^{2}(v,\mu)\|_{0}
\eqsim
|\!|\!|(v,\mu)|\!|\!|_{2,\mathrm{CR}},
\qquad
\forall\,(v,\mu)\in
\mathring V_{k,h}^{\mathrm{VE}}\times
\mathbb P_k(\mathring{\mathcal E}_h).
\]
It remains to prove the equivalence between
\(|\!|\!|\cdot|\!|\!|_{2,h}\) and
\(|\!|\!|\cdot|\!|\!|_{2,\mathrm{CR}}\).

Let \(\delta_h:=v-v^{\mathrm{CR}}\). Then
\(Q_{0,h}\nabla_h\delta_h=0\) by
\eqref{eq:vem-hdiv-commuting}. Hence, by the inverse inequality, the
Poincar\'e inequality, and the norm equivalence
\eqref{eq:vem-grad-projection-equivalence},
\[
|Q_{k-1,h}\nabla_h\delta_h|_{1,h}^2
\eqsim
\sum_{T\in\mathcal T_h}
h_T^{-2}\|Q_{k-1,T}\nabla\delta_h\|_{0,T}^2
\eqsim
\sum_{T\in\mathcal T_h}
h_T^{-2}\|\nabla\delta_h\|_{0,T}^2 .
\]
Together with the inverse inequality and the estimate for the
nonconforming linear interpolation \(v^{\mathrm{CR}}\), this gives
\[
|Q_{k-1,h}\nabla_h\delta_h|_{1,h}^2
\eqsim
\sum_{T\in\mathcal T_h}
h_T^{-4}\|\delta_h\|_{0,T}^2 .
\]
Using the \(L^2\)-norm equivalence \eqref{eq:vem-norm-equiv}, we obtain
\begin{equation}\label{eq:delta-part-equivalence}
	|Q_{k-1,h}\nabla_h\delta_h|_{1,h}^2 \eqsim \sum_{T\in\mathcal T_h}
\Bigl(
h_T^{-4}\|Q_{k-2,T}\delta_h\|_{0,T}^2
+\sum_{F\in\Delta_{d-1}(T)}
h_T^{-3}\|Q_{k-1,F}\delta_h\|_{0,F}^2
\Bigr).
\end{equation}

Moreover, the trace inequality, \(Q_{0,h}\nabla_h\delta_h=0\), and the
Poincar\'e inequality imply
\[
\sum_{F\in\mathcal F_h}
h_F^{-1}
\|[\![Q_{k-1,h}\nabla_h\delta_h]\!]\|_{0,F}^2
\lesssim
|Q_{k-1,h}\nabla_h\delta_h|_{1,h}^2=|Q_{k-1,h}\nabla_hv|_{1,h}^2.
\]
Hence,
	\begin{equation}\label{eq:stable-decomposition-projected-H2}
		\begin{aligned}
			&|Q_{k-1,h}\nabla_h v|_{1,h}^2
			+
			\sum_{F\in\mathcal F_h}
			h_F^{-1}
			\|[\![Q_{k-1,h}\nabla_h v]\!]\|_{0,F}^2 \\
			&\qquad\eqsim
			|Q_{k-1,h}\nabla_h\delta_h|_{1,h}^2
			+
			\sum_{F\in\mathcal F_h}
			h_F^{-1}
			\|[\![\nabla_h v^{\mathrm{CR}}]\!]\|_{0,F}^2 .
		\end{aligned}
	\end{equation}

	For the Crouzeix--Raviart part, we split the gradient jump into its normal
	and tangential components on each face. 
Using the inverse inequality on face $F$ and \eqref{eq:CRjump-by-edge-local},
we have
\begin{equation*}
\sum_{F\in\mathcal F_h}
h_F^{-1}\|\nabla_F[\![v^{\mathrm{CR}}]\!]\|_{0,F}^2 \lesssim
\sum_{T\in\mathcal T_h}
\sum_{e\in\Delta_{d-2}(T)}
h_T^{-2}
\|Q_{k,e}(v^{\mathrm{CR}}-\mu)\|_{0,e}^2.
\end{equation*}
By $\|[\![\nabla_h v^{\mathrm{CR}}]\!]\|_{0,F}^2= \|[\![\partial_n v^{\mathrm{CR}}]\!]\|_{0,F}^2 +\|\nabla_F[\![v^{\mathrm{CR}}]\!]\|_{0,F}^2$,
	it follows that
	\begin{equation}\label{eq:CR-part-equivalence}
		\begin{aligned}
			&\sum_{F\in\mathcal F_h}
			h_F^{-1}\|[\![\nabla_h v^{\mathrm{CR}}]\!]\|_{0,F}^2
			+
			\sum_{T\in\mathcal T_h}
			\sum_{e\in\Delta_{d-2}(T)}
			h_T^{-2}
			\|Q_{k,e}(v^{\mathrm{CR}}-\mu)\|_{0,e}^2 \\
			&\qquad\eqsim
			\sum_{F\in\mathcal F_h}
			h_F^{-1}\|[\![\partial_n v^{\mathrm{CR}}]\!]\|_{0,F}^2
			+
			\sum_{T\in\mathcal T_h}
			\sum_{e\in\Delta_{d-2}(T)}
			h_T^{-2}
			\|Q_{k,e}(v^{\mathrm{CR}}-\mu)\|_{0,e}^2 .
		\end{aligned}
	\end{equation}
Combining \eqref{eq:stable-decomposition-projected-H2} and
\eqref{eq:CR-part-equivalence}, we find
\begin{align*}
|\!|\!|(v,\mu)|\!|\!|_{2,h}^2 &\eqsim |Q_{k-1,h}\nabla_h\delta_h|_{1,h}^2 + 
\sum_{F\in\mathcal F_h}
h_F^{-1}\|[\![\partial_n v^{\mathrm{CR}}]\!]\|_{0,F}^2 \\
&\quad
+
\sum_{T\in\mathcal T_h}
\sum_{e\in\Delta_{d-2}(T)}
h_T^{-2}\|Q_{k,e}(v^{\mathrm{CR}}-\mu)\|_{0,e}^2 .
\end{align*}
The desired equivalence between
\(|\!|\!|\cdot|\!|\!|_{2,h}\) and
\(|\!|\!|\cdot|\!|\!|_{2,\mathrm{CR}}\) now follows from
\eqref{eq:delta-part-equivalence}.
\end{proof}

\section{Symmetric-Tensor Distributional Mixed Methods}\label{sec:distribmfem}

This section presents a distributional mixed method for
\eqref{eq:fourthorderequation}. The method couples normal--normal continuous symmetric tensor finite elements with an \(H^1\)-nonconforming virtual element space and a codimension-two multiplier. We identify its two-dimensional specialization with an HHJ-type formulation and prove optimal-order and parameter-robust a priori error estimates.

\subsection{Distributional mixed method}
For \(k\ge1\), recall from \cite[Section~4.1]{ChenHuang2025} the weak
div\,div operator \((\div\div)_w\) from \(\Sigma_{k,h}^{-1}\) to
\(\mathring{M}^{-1}_{k-2,k-1,k,k}\):
\[
(\div\div)_w\bs \sigma := ((\div\div)_T\bs \sigma, -h_F^{-1}[\tr_2(\bs \sigma)]|_F, h_F^{-3}[\bs n^{\intercal}\bs \sigma\bs n]|_F, h_e^{-2}[ \tr_e(\bs \sigma)]|_e),
\]
where $\tr_2(\bs \sigma)$ and $\tr_e(\bs \sigma)$ are defined in \eqref{eq:divdiv-traces}, and the broken spaces are
\begin{align*}
\Sigma_{k,h}^{-1}&:=\mathbb P_k(\mathcal T_h;\mathbb S),
\quad
\mathring{M}_{k-2,k-1,k,k}^{-1}
:=
\mathbb P_{k-2}(\mathcal T_h)\times
\mathbb P_{k-1}(\mathring{\mathcal F}_h)\times \mathbb P_k(\mathring{\mathcal F}_h)
\times \mathbb P_k(\mathring{\mathcal E}_h).
\end{align*}
For the space \(\mathring{M}_{k-2,k-1,k,k}^{-1}\), define the weighted inner product
\begin{align*}    
((u_0, u_b, u_n, u_e), (v_0, v_b, v_n, v_e))_{0,h} &:=  \sum_{T\in \mathcal T_h}(u_0, v_0)_{T} + \sum_{F\in\mathcal F_h}h_F(u_b, v_b)_{F} \\
&\quad+ \sum_{F\in\mathcal F_h}h_F^3(u_n, v_n)_{F}+ \sum_{e\in\mathcal E_h}h_e^2(u_e, v_e)_{e}.
\end{align*}
The scaling makes all components dimensionally consistent with the cellwise \(L^2\) inner product \((u_0,v_0)\). Thus, for \(\boldsymbol{\sigma}\in \Sigma_{k,h}^{-1}\) and \(v=(v_0,v_b,v_n,v_e)\in \mathring{M}_{k-2,k-1,k,k}^{-1}\),
\begin{align*}
((\div\div)_w\bs \sigma, v)_{0,h} = & \sum_{T\in \mathcal T_h}\big((\div\div \bs \sigma, v_0)_T - (\tr_2(\bs \sigma), v_b)_{\partial T}\big) \\
& + \sum_{T\in \mathcal T_h}(\bs n^{\intercal}
\bs \sigma\bs n,v_n\bs n_F\cdot\bs n)_{\partial T}+ \sum_{T\in \mathcal T_h}\sum_{e\in \Delta_{d-2}(T)}(\tr_e(\bs \sigma), v_e)_e. 
\end{align*}
If \(\boldsymbol{\sigma}\) is normal--normal continuous across interelement faces, i.e., \(\boldsymbol{\sigma} \in \Sigma^{nn}(\mathcal T_h;\mathbb S)\), this reduces to
\begin{align*}
((\div\div)_w\bs \sigma, v)_{0,h} = & \sum_{T\in \mathcal T_h}\big((\div\div \bs \sigma, v_0)_T - (\tr_2(\bs \sigma), v_b)_{\partial T}\big) \\
& + \sum_{T\in \mathcal T_h}\sum_{e\in \Delta_{d-2}(T)}(\tr_e(\bs \sigma), v_e)_e. 
\end{align*}

By the unisolvence of the DoFs \eqref{eq:vem-dofs} for
\(V_k^{\mathrm{VE}}(T)\), the map
\((Q_{k-2,T}, Q_{k-1,F})_{T,F}\) is an isomorphism from
\(\mathring{V}_{k,h}^{\mathrm{VE}}\) onto
\(\mathbb P_{k-2}(\mathcal T_h)\times
\mathbb P_{k-1}(\mathring{\mathcal F}_h)\).
This motivates the following distributional mixed method for solving the fourth-order elliptic singular perturbation problem
 \eqref{eq:fourthorderequation}, i.e. a discretization of the mixed formulation \eqref{eq:mixform}: find \(\boldsymbol{\sigma}_h \in \Sigma_{k,h}^{nn}\), \(u_h \in \mathring{V}_{k,h}^{\mathrm{VE}}\), and \(\lambda_h \in \mathbb P_k(\mathring{\mathcal E}_h)\) such that
 \begin{subequations}\label{eq:fourth-order-discrete-mixed}
 \begin{align}
 a(\boldsymbol{\sigma}_h, \boldsymbol{\tau}_h)
 + b_h(\boldsymbol{\tau}_h;u_h,\lambda_h)
 &= 0, \qquad\qquad\forall\,\boldsymbol{\tau}_h \in \Sigma_{k,h}^{nn},
 \label{eq:fourth-order-discrete-mixed-a}\\
 b_h(\boldsymbol{\sigma}_h; v_h,\mu_h) - c_h(u_h, v_h)
 &= -(f, \tilde{v}_h),\quad\forall\, v_h \in \mathring{V}_{k,h}^{\mathrm{VE}}, \mu_h\in\mathbb{P}_{k}(\mathring{\mathcal{E}}_h),
 \label{eq:fourth-order-discrete-mixed-b}
 \end{align}
 \end{subequations}
where \(\tilde{v}_h := Q_{k-2,h} v_h + (I-Q_{k-2,h}) v_h^{\mathrm{CR}}\), and
 \begin{align*}
 b_h(\boldsymbol{\tau}_h;v_h,\mu_h)
 &:=
 -\sum_{T\in\mathcal T_h}\Bigl[
 (\div\div\boldsymbol{\tau}_h, Q_{k-2,T} v_h)_T
 -(\tr_2(\boldsymbol{\tau}_h), Q_{k-1,F}v_h)_{\partial T} \\
 &\qquad\qquad\quad
 +\sum_{e\in\Delta_{d-2}(T)}(\tr_e(\boldsymbol{\tau}_h), \mu_h)_e
 \Bigr], \\
 c_h(u_h,v_h)
 &:= (Q_{k-1,h}\nabla_h u_h,Q_{k-1,h}\nabla_h v_h).
 \end{align*}

\begin{remark}\label{rem:symmetric-tensor}
\rm
	The use of the symmetric tensor space in
	\eqref{eq:fourth-order-discrete-mixed} reflects both the physical interpretation
	and the mathematical structure of the problem. Physically,
	\(\boldsymbol{\sigma}\) may be interpreted as a scaled bending-moment tensor in a
	Kirchhoff--Love plate model; see, e.g., \cite{Ciarlet1997Plates}. The symmetry
	of such stress or moment tensors is consistent with the symmetry of the Cauchy
	stress tensor in continuum mechanics, which follows from the balance of angular
	momentum; see, e.g., \cite{GurtinFriedAnand2010}. Mathematically, the mixed
	variable is introduced through
	\(\varepsilon^{-2}\boldsymbol{\sigma}=\nabla^2u\), and the Hessian is inherently
	symmetric. Consequently, the skew-symmetric part of a full matrix-valued tensor
	does not contribute to the Hessian coupling in the present formulation.
	Retaining this symmetry preserves the intrinsic structure of the fourth-order
	operator and reduces the number of tensor components from \(d^2\) to
	\(d(d+1)/2\). We note that mixed methods based on full matrix-valued tensor
	variables have also been developed; see
	\cite{BehrensGuzman2011,HuangTang2025}. In contrast, the present method builds
	the symmetry of the tensor variable directly into the tensor space at both the
	continuous and discrete levels.
\end{remark}

By the definition of the weak Hessian $\nabla_w^2$, 
\[
b_h(\boldsymbol{\tau};v,\mu) = -(\boldsymbol{\tau}, \nabla_w^2 (v, \mu)),\quad\forall\,\boldsymbol{\tau} \in \Sigma_{k,h}^{nn}, v \in \mathring H_1^{\mathrm{nc}}(\mathcal T_h), \mu \in L^2(\mathring{\mathcal E}_h).
\]
As an immediate consequence of \eqref{eq:weakHessian-commute}, the following commutative property holds:
\begin{equation}\label{eq:weakdivdivinterproperty}
b_h(\boldsymbol{\tau};v - I_h^{\mathrm{VE}} v, \mu - Q_{k,\mathcal E_h} \mu) = 0, 
\quad \forall\,\boldsymbol{\tau} \in \Sigma_{k,h}^{nn}.
\end{equation}

We next prove the well-posedness of the distributional mixed method \eqref{eq:fourth-order-discrete-mixed}.

\begin{theorem}\label{thm:fourth-orderdistmixwellposedness}
The distributional mixed method \eqref{eq:fourth-order-discrete-mixed} is well-posed.
\end{theorem}
\begin{proof}
It suffices to prove that the distributional mixed method \eqref{eq:fourth-order-discrete-mixed} has only the zero solution when \(f=0\). 

Subtracting \eqref{eq:fourth-order-discrete-mixed-b} with \((v_h,\mu_h)=(u_h,\lambda_h)\) from \eqref{eq:fourth-order-discrete-mixed-a} with \(\boldsymbol{\tau}_h=\boldsymbol{\sigma}_h\) gives
\[
\varepsilon^{-2}\|\boldsymbol{\sigma}_h\|_0^2 + \|Q_{k-1,h}\nabla_h u_h\|_0^2=0.
\]
Together with the norm equivalence \eqref{eq:vem-grad-projection-equivalence}, this identity yields $\boldsymbol{\sigma}_h=0$ and $u_h=0$.
Then \eqref{eq:fourth-order-discrete-mixed-a} reduces to 
\[
(\boldsymbol{\tau}, \nabla_w^2(u_h,\lambda_h))=0,\quad\forall\,\boldsymbol{\tau} \in \Sigma_{k,h}^{nn}.
\]
Hence $\nabla_w^{2} (u_h,\lambda_h)=0$, and the norm equivalence \eqref{eq:normequivalence} gives $\lambda_h=0$.
\end{proof}

\subsection{Relation to the Hellan--Herrmann--Johnson method}
Now we discuss the relation between the distributional mixed method \eqref{eq:fourth-order-discrete-mixed} and the classical Hellan--Herrmann--Johnson (HHJ) method \cite{Hellan1967,Herrmann1967,Johnson1973} for solving the fourth-order elliptic singular perturbation problem \eqref{eq:fourthorderequation} in two dimensions.

Let the two-dimensional Lagrange finite element space of degree $k+1$ with homogeneous boundary conditions be
\begin{equation*}
 \mathring{V}_{k+1,h}^{L}
 :=
 \{ w_h\in H_0^1(\Omega): w_h|_T\in \mathbb{P}_{k+1}(T)\quad \forall\,T\in\mathcal{T}_h\}.
\end{equation*}
The space $\mathring{V}_{k,h}^{\mathrm{VE}}\times \mathbb{P}_k(\mathring{\mathcal{E}}_h)$ is naturally isomorphic to $\mathring{V}_{k+1,h}^{L}$. This isomorphism is the key step in identifying \eqref{eq:fourth-order-discrete-mixed} with a two-dimensional HHJ-type method.

\begin{lemma}\label{lem:hhj-isomorphism}
Let $d=2$ and $k\ge 1$. For every $(v_h,\mu_h)\in \mathring{V}_{k,h}^{\mathrm{VE}} \times \mathbb{P}_k(\mathring{\mathcal{E}}_h)$, there exists a unique function $w_h\in \mathring{V}_{k+1,h}^{L}$ such that
\begin{equation}\label{eq:whvhmuh}
 Q_{k-2,T}\, w_h = Q_{k-2,T}\, v_h,\qquad
 Q_{k-1,F}\, w_h = Q_{k-1,F}\, v_h,\qquad
 w_h(e) = \mu_h(e),
\end{equation}
 for all $T \in \mathcal{T}_h$, $F \in \mathring{\mathcal{F}}_h$, and $e \in \mathring{\mathcal{E}}_h$. This defines an isomorphism
 \[
 \mathcal{I}_h : \mathring{V}_{k,h}^{\mathrm{VE}} \times \mathbb{P}_k(\mathring{\mathcal{E}}_h) \to \mathring{V}_{k+1,h}^{L},
 \]
and it holds that
\begin{equation}\label{eq:QhgradVeL}
Q_{k-1,h}\nabla w_h=Q_{k-1,h}\nabla_h v_h, \quad\forall\,v_h\in \mathring{V}_{k,h}^{\mathrm{VE}}.
\end{equation}
\end{lemma}

\begin{proof}
The well-posedness of the isomorphism $\mathcal{I}_h$ follows directly from the definitions of $\mathring{V}_{k,h}^{\mathrm{VE}}$ and $\mathring{V}_{k+1,h}^{L}$.
The identity \eqref{eq:QhgradVeL} follows from the integration by parts and \eqref{eq:whvhmuh}.
\end{proof}

\begin{lemma}
Let $d=2$ and $k\ge 1$. For any $\boldsymbol{\tau}\in \Sigma_{k,h}^{nn}$ and $(v,\mu)\in \mathring{V}_{k,h}^{\mathrm{VE}}\times \mathbb{P}_k(\mathring{\mathcal{E}}_h)$, let $w=\mathcal I_h(v,\mu)\in \mathring{V}_{k+1,h}^{L}$. Then
\begin{equation}\label{eq:bhbhHHJ}
 b_h(\boldsymbol{\tau};v,\mu)
 =
 b_h^{\mathrm{HHJ}}(\boldsymbol{\tau},w),
\end{equation}
where the bilinear form $b_h^{\mathrm{HHJ}}$ is defined by
\begin{equation*}
 b_h^{\mathrm{HHJ}}(\boldsymbol{\tau},w)
 :=
 -\sum_{T\in\mathcal{T}_h} (\boldsymbol{\tau},\nabla^2 w)_T
 +
 \sum_{F\in\mathcal{F}_h} (\bs n^{\intercal}\boldsymbol{\tau}\bs n,[\![\partial_{n_F}w]\!])_F.
\end{equation*}
\end{lemma}

\begin{proof}
 By the definition of $w$, 
 \begin{equation*}
 b_h(\boldsymbol{\tau}; v,\mu)
 =
 -
 \sum_{T\in\mathcal{T}_h}
 \big(
 (\div\div\boldsymbol{\tau}, w)_T
 -
 (\tr_2(\boldsymbol{\tau}), w)_{\partial T}
 +
 \sum_{e\in\Delta_{d-2}(T)} (\tr_e(\boldsymbol{\tau}), w)_e
 \big).
 \end{equation*}
 Applying the Green's identity \eqref{eq:greenidentitydivdiv} on each element $T$, with $w|_T\in \mathbb P_{k+1}(T)$, yields
 \begin{equation*}
 \begin{aligned}
 (\div\div\boldsymbol{\tau}, w)_T
 &- (\tr_2(\boldsymbol{\tau}), w)_{\partial T}
 + \sum_{e\in\Delta_{d-2}(T)} (\tr_e(\boldsymbol{\tau}), w)_e
 \\
 &= (\boldsymbol{\tau},\nabla^2 w)_T
 - (\bs n^{\intercal}\boldsymbol{\tau}\bs n,\partial_n w)_{\partial T}.
 \end{aligned}
 \end{equation*}
A combination of the last two equations yields
 \begin{equation*}
 b_h(\boldsymbol{\tau}; v,\mu)
 =
 -
 \sum_{T\in\mathcal{T}_h}
 \left[
 (\boldsymbol{\tau},\nabla^2 w)_T
 -
 (\bs n^{\intercal}\boldsymbol{\tau}\bs n,\partial_{n} w)_{\partial T}
 \right].
 \end{equation*}
 Since $\boldsymbol{\tau}\in \Sigma_{k,h}^{nn}$, its normal--normal component $\bs n^{\intercal}\boldsymbol{\tau}\bs n$ is single-valued across each interior face. Therefore, summing the boundary contributions elementwise and regrouping them over the mesh faces gives \eqref{eq:bhbhHHJ}.
\end{proof}

Thus, under the isomorphism $\mathcal I_h$, the bilinear form $b_h(\boldsymbol{\tau};v,\mu)$ coincides in two dimensions with the HHJ bilinear form $b_h^{\mathrm{HHJ}}(\boldsymbol{\tau},\mathcal I_h(v,\mu))$.

\begin{theorem}\label{thm:hhj-equivalence}
Let $d=2$ and $k\ge 1$. Let \((\boldsymbol{\sigma}_h,u_h,\lambda_h)\in \Sigma_{k,h}^{nn} \times
\mathring{V}_{k,h}^{\mathrm{VE}}\times\mathbb{P}_{k}(\mathring{\mathcal{E}}_h)\)
be the solution of the distributional mixed method \eqref{eq:fourth-order-discrete-mixed}. Then $(\boldsymbol{\sigma}_h,w_h)\in \Sigma_{k,h}^{nn}\times \mathring{V}_{k+1,h}^{L}$ with $w_h=\mathcal I_h(u_h,\lambda_h)$ satisfies the following HHJ-type method:
\begin{subequations}\label{eq:fourth-order-HHJtype}
\begin{align}
\label{eq:fourth-order-HHJtype1}
 \varepsilon^{-2}(\boldsymbol{\sigma}_h,\boldsymbol{\tau}_h)
 + b_h^{\mathrm{HHJ}}(\boldsymbol{\tau}_h,w_h)
 &= 0,\qquad\qquad\forall\,\boldsymbol{\tau}_h \in \Sigma_{k,h}^{nn},\\
\label{eq:fourth-order-HHJtype2}
 b_h^{\mathrm{HHJ}}(\boldsymbol{\sigma}_h,\chi_h)
 - (Q_{k-1,h}\nabla w_h,Q_{k-1,h}\nabla \chi_h)
 &= -(f, \tilde{\chi}_h),\quad \forall\,\chi_h \in \mathring{V}_{k+1,h}^{L},
\end{align} 
\end{subequations}
where
$\tilde{\chi}_h=Q_{k-2,h}\chi_h + (I-Q_{k-2,h})\chi_h^{\mathrm{CR}}$.
\end{theorem}
\begin{proof}
Set $\chi_h=\mathcal I_h(v_h,\mu_h)$ for $v_h\in\mathring{V}_{k,h}^{\mathrm{VE}}$ and $\mu_h\in\mathbb{P}_{k}(\mathring{\mathcal{E}}_h)$.
By \eqref{eq:whvhmuh}-\eqref{eq:bhbhHHJ} and the fact $\chi_h^{\mathrm{CR}}=v_h^{\mathrm{CR}}$, the distributional mixed method \eqref{eq:fourth-order-discrete-mixed} can be recast as
\begin{align*}
 \varepsilon^{-2}(\boldsymbol{\sigma}_h,\boldsymbol{\tau}_h)
 + b_h^{\mathrm{HHJ}}(\boldsymbol{\tau}_h,w_h)
 &= 0, \qquad\quad\;\;\forall\,\boldsymbol{\tau}_h \in \Sigma_{k,h}^{nn},\\
 b_h^{\mathrm{HHJ}}(\boldsymbol{\sigma}_h,\chi_h)
 - (Q_{k-1,h}\nabla w_h,Q_{k-1,h}\nabla \chi_h)
 &= -(f, \tilde{\chi}_h),\;\;\forall\, v_h \in \mathring{V}_{k,h}^{\mathrm{VE}}, \mu_h\in\mathbb{P}_{k}(\mathring{\mathcal{E}}_h).
\end{align*} 
The equivalence between \eqref{eq:fourth-order-discrete-mixed} and the HHJ-type method \eqref{eq:fourth-order-HHJtype} therefore follows from the isomorphism $\mathcal I_h$.
\end{proof}

For comparison, the HHJ method for the two-dimensional problem \eqref{eq:fourthorderequation} considered in \cite{LiuHuangWang2020} reads as follows: find $(\boldsymbol{\sigma}_h,w_h)\in \Sigma_{k,h}^{nn}\times \mathring{V}_{k+1,h}^{L}$ such that
\begin{subequations}\label{eq:fourth-order-HHJ}
\begin{align}
\label{eq:fourth-order-HHJ1}
 \varepsilon^{-2}(\boldsymbol{\sigma}_h,\boldsymbol{\tau}_h)
 + b_h^{\mathrm{HHJ}}(\boldsymbol{\tau}_h,w_h)
 &= 0,
 &&\forall\,\boldsymbol{\tau}_h\in \Sigma_{k,h}^{nn},\\
\label{eq:fourth-order-HHJ2}
 b_h^{\mathrm{HHJ}}(\boldsymbol{\sigma}_h,v_h)
 - (\nabla w_h,\nabla v_h)
 &= -(f,v_h),
 &&\forall\,v_h\in \mathring{V}_{k+1,h}^{L}.
\end{align}
\end{subequations}

Compared with the HHJ method \eqref{eq:fourth-order-HHJ}, the HHJ-type formulation \eqref{eq:fourth-order-HHJtype} uses the projected gradient term 
\[
(Q_{k-1,h}\nabla w_h,Q_{k-1,h}\nabla \chi_h)=(Q_{k-1,h}\nabla_hu_h,Q_{k-1,h}\nabla_hv_h)
\] 
instead of $(\nabla w_h,\nabla v_h)$, and it uses the modified load term $(f,\tilde{\chi}_h)$ associated with the nonconforming scalar variable. The fourth-order HHJ bilinear form is unchanged after identifying $w_h$ with $(v_h,\mu_h)$ through $\mathcal I_h$. Thus the difference lies in the discretization of the second-order term: \eqref{eq:fourth-order-HHJtype}, equivalently \eqref{eq:fourth-order-discrete-mixed}, uses an $H^1$-nonconforming scalar space, whereas \eqref{eq:fourth-order-HHJ} uses an $H^1$-conforming Lagrange space.

The classical HHJ method with a Lagrange scalar space is intrinsically two-dimensional. By replacing that scalar space with the pair $\mathring{V}_{k,h}^{\mathrm{VE}}\times\mathbb P_k(\mathring{\mathcal E}_h)$ and retaining the normal--normal continuous tensor space $\Sigma_{k,h}^{nn}$, the distributional mixed method \eqref{eq:fourth-order-discrete-mixed} extends the HHJ framework to arbitrary spatial dimension.

 \subsection{Error analysis}
 We equip \(\mathring{V}_{k,h}^{\mathrm{VE}}\times \mathbb{P}_{k}(\mathring{\mathcal{E}}_h)\)
 with the parameter-dependent norm
 \[
 |\! |\!|(v_h,\mu_h)|\!|\!|_{\varepsilon,h}^2
 := \varepsilon^2|\! |\!|(v_h,\mu_h)|\!|\!|_{2,h}^2 + | v_h |_{1,h}^2.
 \]

Let \((\boldsymbol{\sigma},u)\in
H^{-1}(\operatorname{div}\operatorname{div},\Omega;\mathbb S)
\times H_0^1(\Omega)\) be the solution of \eqref{eq:mixform}.
Whenever the codimension-two traces of \(u\) are well-defined, we define
\(\lambda\in L^2(\mathring{\mathcal E}_h)\) by
\[
\lambda|_e:=u|_e\qquad \forall\, e\in\mathring{\mathcal E}_h,
\]
and set \(\boldsymbol{\sigma}_I:=I_h^{nn}\boldsymbol{\sigma}\) whenever \(I_h^{nn}\boldsymbol\sigma\) is well-defined.
Define the residual functional
\[
\mathcal R_h(v_h,\mu_h)
:=
b_h(\boldsymbol\sigma_I;v_h,\mu_h)
-
c_h(u,v_h)
+
(f,\widetilde v_h).
\]

 \begin{lemma}\label{lem:Consistency-mixed1}
 Let \((\boldsymbol{\sigma},u)\in H^{-1}(\div\div,\Omega;\mathbb S)\times
H_0^2(\Omega)\) be the solution of \eqref{eq:mixform}.
Then
 \begin{equation}\label{eq:Consistency-mixed1}
 a(\boldsymbol{\sigma},\boldsymbol{\tau}_h)
 + b_h(\boldsymbol{\tau}_h; I_h^{\mathrm{VE}}u,Q_{k,\mathcal E_h}\lambda) = 0
 \qquad \forall\, \boldsymbol{\tau}_h \in \Sigma_{k,h}^{nn}.
 \end{equation}
 \end{lemma}
 \begin{proof}
Since \(\boldsymbol\sigma=\varepsilon^2\nabla^2 u\), we have
\(a(\boldsymbol\sigma,\boldsymbol\tau_h)=(\nabla^2 u,\boldsymbol\tau_h)\).
An elementwise integration by parts gives
\[
(\nabla^2 u,\boldsymbol\tau_h)
=
-b_h(\boldsymbol\tau_h;u,\lambda),
\]
and hence
\[
a(\boldsymbol\sigma,\boldsymbol\tau_h)
+
b_h(\boldsymbol\tau_h;u,\lambda)
=0
\qquad
\forall\,\boldsymbol\tau_h\in\Sigma_{k,h}^{nn}.
\]
The desired identity follows from the interpolation property
\eqref{eq:weakdivdivinterproperty}.
\end{proof}

We next estimate the consistency residual associated with
\eqref{eq:fourth-order-discrete-mixed-b}.

\begin{lemma}\label{lem:Compatibilityerror}
Let
$
(\boldsymbol\sigma,u)\in
H^{-1}(\operatorname{div}\operatorname{div},\Omega;\mathbb S)
\times H_0^1(\Omega)
$
be the solution of \eqref{eq:mixform}. Assume that
\(\boldsymbol\sigma\in H^{k+1}(\Omega;\mathbb S)\) and
\(u\in H^{k+1}(\Omega)\). Then, for \(k\ge1\),
 \begin{equation}\label{eq:Compatibilityerror}
|\mathcal R_h(v_h,\mu_h)|
 \lesssim
 h^{k+1}\bigl(|\boldsymbol{\sigma}|_{k+1}+|u|_{k+1}\bigr)\,
 |\!|\!|(v_h,\mu_h)|\!|\!|_{2,h} \;\;\forall\,(v_h,\mu_h)\in
\mathring V_{k,h}^{\mathrm{VE}}\times
\mathbb P_k(\mathring{\mathcal E}_h).
 \end{equation}
\end{lemma}
 
 \begin{proof}
 Using \(f=\div\div\boldsymbol{\sigma}-\Delta u\), we decompose
 \(\mathcal R_h(v_h,\mu_h)=I_1+I_2\), where
 \[
 I_1:=b_h(\boldsymbol{\sigma}_I; v_h,\mu_h)
 +(\div\div\boldsymbol{\sigma},\tilde v_h),
 \qquad
 I_2:=-c_h(u,v_h)-(\Delta u,\tilde v_h).
 \]
By the norm equivalence \eqref{eq:normequivalence}, it suffices to prove
\begin{align}
\label{eq:est-I1}
I_1&\lesssim h^{k+1}|\boldsymbol{\sigma}|_{k+1}
 |\!|\!|(v_h,\mu_h)|\!|\!|_{2,\mathrm{CR}}, \\
\label{eq:est-I2}
 I_2& \lesssim
 h^{k+1}|u|_{k+1}|\!|\!|(v_h,\mu_h)|\!|\!|_{2,\mathrm{CR}}.
\end{align}

We first bound \(I_1\). Since
\[
(\div\div\boldsymbol{\sigma}_I,
 \tilde v_h-Q_{k-2,T}v_h)_T
=
(\div\div\boldsymbol{\sigma}_I,
 (I-Q_{k-2,T})v_h^{\mathrm{CR}})_T
=0,
\]
and since \(Q_{k-1,F}v_h\) and \(\mu_h\) are single-valued on
interelement faces and codimension-two subsimplices, respectively, we obtain
\begin{align*}
 I_1
 &=
 \sum_{T\in\mathcal T_h}\Bigl[
  (\div\div(\boldsymbol{\sigma}-\boldsymbol{\sigma}_I), \tilde v_h)_T + (\tr_2(\boldsymbol{\sigma}_I), Q_{k-1,F}v_h)_{\partial T}
  \\
 &\qquad\qquad
 -\sum_{e\in\Delta_{d-2}(T)}
 (\tr_e(\boldsymbol{\sigma}_I), \mu_h)_e
 \Bigr] \\
 &=
 \sum_{T\in\mathcal T_h}\Bigl[
  (\div\div(\boldsymbol{\sigma}-\boldsymbol{\sigma}_I), \tilde v_h)_T - (\tr_2(\boldsymbol{\sigma}-\boldsymbol{\sigma}_I), Q_{k-1,F}v_h)_{\partial T}
  \\
 &\qquad\qquad
 +\sum_{e\in\Delta_{d-2}(T)}
 (\tr_e(\boldsymbol{\sigma}-\boldsymbol{\sigma}_I), \mu_h)_e
 \Bigr].
\end{align*} 
Applying the Green identity \eqref{eq:greenidentitydivdiv} for the
\(\div\div\) operator then yields
\[
I_1=
 \sum_{T\in\mathcal T_h}\Bigl[
 \bigl(\tr_2(\boldsymbol{\sigma}-\boldsymbol{\sigma}_I),
 \tilde v_h-Q_{k-1,F}v_h\bigr)_{\partial T}
 +\sum_{e\in\Delta_{d-2}(T)}
 \bigl(\tr_e(\boldsymbol{\sigma}-\boldsymbol{\sigma}_I),
 \mu_h-\tilde v_h\bigr)_e
 \Bigr].
\]

For \(k\ge2\), using the definition of \(\tilde v_h\) and $Q_{k-1,F}v_h^{\mathrm{CR}}=v_h^{\mathrm{CR}}|_F$, we further have
\[
\begin{aligned}
I_1
&=
 \sum_{T\in\mathcal T_h}\Bigl[
 \bigl(\tr_2(\boldsymbol{\sigma}-\boldsymbol{\sigma}_I),
 (Q_{k-2,T}-Q_{k-1,F})(v_h-v_h^{\mathrm{CR}})\bigr)_{\partial T}
 \\
&\qquad\qquad
 -\sum_{e\in\Delta_{d-2}(T)}
 \bigl(\tr_e(\boldsymbol{\sigma}-\boldsymbol{\sigma}_I),
 Q_{k-2,T}(v_h-v_h^{\mathrm{CR}})
 +v_h^{\mathrm{CR}}-\mu_h\bigr)_e
 \Bigr].
\end{aligned}
\]
The estimate \eqref{eq:est-I1} for \(k\ge2\) then follows from the
Cauchy--Schwarz inequality, standard inverse estimates, and the
approximation property \eqref{eq:intererrorSigma} of \(I_h^{nn}\).

It remains to consider \(k=1\). In this case
\(\tilde v_h=v_h^{\mathrm{CR}}=v_h\). Hence
\[
\begin{aligned}
I_1
&=
 \sum_{T\in\mathcal T_h}\Bigl[
 \bigl(\tr_2(\boldsymbol{\sigma}-\boldsymbol{\sigma}_I),
 v_h-Q_{0,F}v_h\bigr)_{\partial T}
 -\sum_{e\in\Delta_{d-2}(T)}
 \bigl(\tr_e(\boldsymbol{\sigma}-\boldsymbol{\sigma}_I),
 v_h-\mu_h\bigr)_e
 \Bigr]
\\
&=
 \sum_{T\in\mathcal T_h}\Bigl[
 \bigl(\tr_2\boldsymbol{\sigma}
       -Q_{0,F}(\tr_2\boldsymbol{\sigma}),v_h\bigr)_{\partial T}
 -\sum_{e\in\Delta_{d-2}(T)}
 \bigl(\tr_e(\boldsymbol{\sigma}-\boldsymbol{\sigma}_I),
 v_h-\mu_h\bigr)_e
 \Bigr]
\\
&=
 \sum_{F\in\mathcal F_h}
 \bigl(\tr_2\boldsymbol{\sigma}
       -Q_{0,F}(\tr_2\boldsymbol{\sigma}),[\![v_h]\!]\bigr)_F
 -\sum_{T\in\mathcal T_h}\sum_{e\in\Delta_{d-2}(T)}
 \bigl(\tr_e(\boldsymbol{\sigma}-\boldsymbol{\sigma}_I),
 v_h-\mu_h\bigr)_e.
\end{aligned}
\]
Using the Cauchy--Schwarz inequality, \eqref{eq:intererrorSigma}, and estimate
\eqref{eq:CRjump-by-edge-local}, we obtain \eqref{eq:est-I1} for \(k=1\).

We now estimate \(I_2\). Suppose first that \(k\ge3\). Then
\(\tilde v_h=Q_{k-2,h}v_h\). By
\eqref{eq:hdiv-projection-difference} with \(\boldsymbol w=\nabla u\) and
the commuting property \eqref{eq:commutationIhdiv}, we get
\[
\begin{aligned}
 I_2
 &=
 -\sum_{T\in\mathcal T_h}\sum_{F\subset\partial T}
 \bigl((Q_{k-1,T}\nabla u-\nabla u)\cdot\boldsymbol n,
 Q_{k-1,F}v_h-Q_{k-2,T}v_h\bigr)_F
 \\
 &\lesssim
 h^{k+1}|u|_{k+1}
 |\!|\!|(v_h,\mu_h)|\!|\!|_{2,\mathrm{CR}}.
\end{aligned}
\]

For \(k=2\), set \(\delta_h:=v_h-v_h^{\mathrm{CR}}\). Since
\(\tilde v_h=v_h-(I-Q_{0,h})\delta_h\), the \(L^2\)-orthogonality of
\(Q_{1,h}\), together with the weak continuity of \(v_h\), gives
\[
\begin{aligned}
 I_2
 &=
 -(Q_{1,h}\nabla u,\nabla_hv_h)
 -(\Delta u,v_h-(I-Q_{0,h})\delta_h)
 \\
 &=
 ((I-Q_{1,h})\nabla u,\nabla_hv_h)
 +((I-Q_{0,h})\Delta u,\delta_h)
 -\sum_{T\in\mathcal T_h}(\partial_nu,v_h)_{\partial T}
 \\
 &=
 ((I-Q_{1,h})\nabla u,\nabla_h\delta_h)
 +((I-Q_{0,h})\Delta u,\delta_h)
 -\sum_{T\in\mathcal T_h}
 ((I-Q_{1,F})\partial_nu,\delta_h)_{\partial T}
 \\
 &\lesssim
 h^3|u|_3
 |\!|\!|(v_h,\mu_h)|\!|\!|_{2,\mathrm{CR}}.
\end{aligned}
\]

Finally, for \(k=1\), we have \(\tilde v_h=v_h\), and \(v_h\) is
piecewise affine. Therefore, by an elementwise integration by parts and \eqref{eq:CRjump-by-edge-local},
\[
\begin{aligned}
 I_2
 &=
 -(\nabla u,\nabla_hv_h)-(\Delta u,v_h)
 =
 -\sum_{F\in\mathcal F_h}
 \bigl((I-Q_{0,F})\partial_nu,[\![v_h]\!]\bigr)_F
 \\
 &\lesssim
 h^2|u|_2
 |\!|\!|(v_h,\mu_h)|\!|\!|_{2,\mathrm{CR}}.
\end{aligned}
\]
Combining the preceding three estimates yields \eqref{eq:est-I2}. The proof is complete.
\end{proof}
 
For the discrete solution
\((\boldsymbol{\sigma}_h,u_h,\lambda_h)\) of
\eqref{eq:fourth-order-discrete-mixed}, define the errors
\[
e_{\boldsymbol{\sigma}}:=\boldsymbol{\sigma}_I-\boldsymbol{\sigma}_h,
\qquad
e_u:=I_h^{\mathrm{VE}}u-u_h,
\qquad
e_\lambda:=Q_{k,\mathcal E_h}\lambda-\lambda_h .
\]

 \begin{theorem}\label{thm:errorestimate}
 Let \((\boldsymbol{\sigma}, u) \in H^{-1}(\div\div, \Omega; \mathbb{S}) \times
 H_0^1(\Omega)\) and
 \((\boldsymbol{\sigma}_h,u_h,\lambda_h)\in \Sigma_{k,h}^{nn} \times
 \mathring{V}_{k,h}^{\mathrm{VE}}\times\mathbb{P}_{k}(\mathring{\mathcal{E}}_h)\)
 be the solutions of \eqref{eq:mixform} and
 \eqref{eq:fourth-order-discrete-mixed}, respectively. Assume
 \(u \in H^{k+3}(\Omega)\). Then, for \(k\ge1\),
 \begin{equation}\label{eq:errorestimate}
 \varepsilon^{-1}\|\boldsymbol{\sigma} - \boldsymbol{\sigma}_h\|_{0}
 + |\! |\!|(I_h^{\mathrm{VE}}u-u_h,Q_{k,\mathcal{E}_h}\lambda-\lambda_h)|\!|\!|_{\varepsilon,h}
 \lesssim
 h^{k+1}(\varepsilon \| u \|_{k+3} + \varepsilon^{-1}|u|_{k+1}).
 \end{equation}
 \end{theorem}
 
 \begin{proof}
 By \eqref{eq:Consistency-mixed1}, the commutative property \eqref{eq:weakdivdivinterproperty} and
 the interpolation property \eqref{eq:vem-hdiv-commuting} of \(I_h^{\mathrm{VE}}\), we obtain the following error equations from \eqref{eq:fourth-order-discrete-mixed}:
\begin{subequations}\label{eq:error}
\begin{align}
 a(e_{\boldsymbol{\sigma}},\boldsymbol{\tau}_h)
 + b_h(\boldsymbol{\tau}_h; e_u,e_\lambda)
 &=
 a(\boldsymbol{\sigma}_I-\boldsymbol{\sigma},\boldsymbol{\tau}_h),
 && \forall\, \boldsymbol{\tau}_h\in\Sigma_{k,h}^{nn},
 \label{errorequation1}
 \\
 b_h(e_{\boldsymbol{\sigma}}; v_h,\mu_h)
 - c_h(e_u,v_h)
 &=
 \mathcal R_h(v_h,\mu_h),
 && \forall\, (v_h,\mu_h)\in
 \mathring V_{k,h}^{\mathrm{VE}}\times
 \mathbb P_k(\mathring{\mathcal E}_h).
 \label{errorequation2}
\end{align}
\end{subequations}

By the norm equivalence \eqref{eq:normequivalence} and error equation \eqref{errorequation1},
\begin{equation}\label{eq:20260430-1}
\begin{aligned}
|\! |\!|(e_u,e_\lambda)|\!|\!|_{2,h} &\eqsim \|\nabla_w^{2} (e_u,e_\lambda)\|_{0} =\sup_{\boldsymbol{\tau}_h\in \Sigma_{k,h}^{nn}} \frac{b_h(\boldsymbol{\tau}_h; e_u,e_\lambda)}{\|\boldsymbol{\tau}_h\|_{0}} \\
&\lesssim \varepsilon^{-2}\|e_{\boldsymbol{\sigma}}\|_0 + \varepsilon^{-2}\|\boldsymbol{\sigma}_I-\boldsymbol{\sigma}\|_0.
\end{aligned}
\end{equation}
This combined with \eqref{eq:Compatibilityerror} gives
 \begin{equation}\label{eq:Rh-bound}
|\mathcal R_h(e_u,e_\lambda)|
 \lesssim
 h^{k+1}\bigl(|\boldsymbol{\sigma}|_{k+1}+|u|_{k+1}\bigr)\,(\varepsilon^{-2}\|e_{\boldsymbol{\sigma}}\|_0 + \varepsilon^{-2}\|\boldsymbol{\sigma}_I-\boldsymbol{\sigma}\|_0).
 \end{equation}

Taking \(\boldsymbol\tau_h=e_{\boldsymbol\sigma}\) in
\eqref{errorequation1} and
\((v_h,\mu_h)=(e_u,e_\lambda)\) in \eqref{errorequation2}, and then
subtracting the two identities, we obtain
\begin{equation*}
 \varepsilon^{-2}\|e_{\boldsymbol{\sigma}}\|_0^2
 + \|Q_{k-1,h}\nabla_h e_u\|_0^2
 =
 \varepsilon^{-2}(\boldsymbol{\sigma}_I-\boldsymbol{\sigma},
 e_{\boldsymbol{\sigma}})
 - \mathcal R_h(e_u,e_\lambda).
\end{equation*}
Using the Cauchy--Schwarz inequality, \eqref{eq:Rh-bound}, and the
interpolation estimate \eqref{eq:intererrorSigma}, we infer that
\[
\varepsilon^{-1}\|e_{\boldsymbol{\sigma}}\|_0 + \|Q_{k-1,h}\nabla_h e_u\|_0\lesssim \varepsilon^{-1}h^{k+1}(|\boldsymbol{\sigma}|_{k+1}+|u|_{k+1}).
\]
This together with \eqref{eq:20260430-1} and \eqref{eq:intererrorSigma} yields
\[
\varepsilon^{-1}\|e_{\boldsymbol{\sigma}}\|_0 + |\! |\!|(e_u,e_\lambda)|\!|\!|_{\varepsilon,h}\lesssim \varepsilon^{-1}h^{k+1}(|\boldsymbol{\sigma}|_{k+1}+|u|_{k+1}).
\]
Finally, we conclude the desired estimate \eqref{eq:errorestimate} by using the triangle inequality and \eqref{eq:intererrorSigma}.
\end{proof}
 
 We next derive a parameter-robust error estimate for the distributional method
 \eqref{eq:fourth-order-discrete-mixed}.
 \begin{theorem}\label{thm:robusterror}
Let \((\boldsymbol{\sigma}, u) \in H^{-1}(\div\div, \Omega; \mathbb{S}) \times
 H^2(\Omega)\) and
 \((\boldsymbol{\sigma}_h,u_h,\lambda_h)\in \Sigma_{k,h}^{nn} \times
 \mathring{V}_{k,h}^{\mathrm{VE}}\times\mathbb{P}_{k}(\mathring{\mathcal{E}}_h)\)
 be the solutions of \eqref{eq:mixform} and
 \eqref{eq:fourth-order-discrete-mixed}, respectively.
	Assume that regularities \eqref{eq:regularity-assumption} and
	\eqref{eq:perturbation-regularity} hold with \(s=k+1\). Then, for
	\(k\ge1\),
	\begin{align}
		\varepsilon^{-1}\|\boldsymbol{\sigma}-\boldsymbol{\sigma}_h\|_{0}
		+ |\!|\!|(u-u_h,\lambda-\lambda_h)|\!|\!|_{\varepsilon,h}
		&\lesssim \varepsilon^{1/2}\|f\|_0 + h^k\|f\|_{k-1},
		\label{eq:robusterror1} \\
		\varepsilon
		\left|Q_{k-1,h}\nabla_h(\bar u-u_h)\right|_{1,h}
		+ |\bar u-u_h|_{1,h}
		&\lesssim \varepsilon^{1/2}\|f\|_0 + h^k\|f\|_{k-1}.
		\label{eq:robusterror2}
	\end{align}
\end{theorem}

\begin{proof}
By \eqref{eq:vem-hdiv-commuting} and
\(I_{1,h}^{\mathrm{VE}}(u-I_h^{\mathrm{VE}}u)=0\), we have
	\[
	|\!|\!|(u-I_h^{\mathrm{VE}}u,\lambda-Q_{k,\mathcal E_h}\lambda)
	|\!|\!|_{\varepsilon,h}
	=
	|u-I_h^{\mathrm{VE}}u|_{1,h}.
	\]
	Using the triangle inequality, the interpolation estimate \eqref{eq:intererrorVE} and the regularity
	assumptions \eqref{eq:regularity-assumption}-\eqref{eq:perturbation-regularity}, we obtain
	\[
	\begin{aligned}
		|u-I_h^{\mathrm{VE}}u|_{1,h}
		&\le
		|(u-\bar u)-I_h^{\mathrm{VE}}(u-\bar u)|_{1,h}
		+ |\bar u-I_h^{\mathrm{VE}}\bar u|_{1,h} \\
		&\lesssim
		|u-\bar u|_1+h^k|\bar u|_{k+1}
		\lesssim
		\varepsilon^{1/2}\|f\|_0+h^k\|f\|_{k-1}.
	\end{aligned}
	\]
	Thus, by the fact $\varepsilon^{-1}\|\boldsymbol{\sigma}\|_0 = \varepsilon |u|_2 \lesssim \varepsilon^{1/2}\|f\|_0$,
	\begin{equation}\label{eq:robust-pre-est}
		\varepsilon^{-1}\|\boldsymbol{\sigma}\|_0
		+
		|\!|\!|(u-I_h^{\mathrm{VE}}u,\lambda-Q_{k,\mathcal E_h}\lambda)
		|\!|\!|_{\varepsilon,h}
		\lesssim
		\varepsilon^{1/2}\|f\|_0+h^k\|f\|_{k-1}.
	\end{equation}
	
	By error equation \eqref{errorequation1}, we have
	\begin{equation}\label{eq:robust-first-equation}
		-(\boldsymbol{\tau}_h,\nabla_w^{2}(e_u,e_\lambda))=b_h(\boldsymbol{\tau}_h; e_u,e_\lambda)
		=
		\varepsilon^{-2}
		(\boldsymbol{\sigma}_h-\boldsymbol{\sigma},\boldsymbol{\tau}_h),
		\qquad
		\forall\,\boldsymbol{\tau}_h\in\Sigma_{k,h}^{nn}.
	\end{equation}
	Therefore, by the norm equivalence \eqref{eq:normequivalence} and \eqref{eq:robust-first-equation},
	\begin{equation}\label{eq:robust-H2-control}
		|\!|\!|(e_u,e_\lambda)|\!|\!|_{2,h}
		\eqsim \|\nabla_w^{2}(e_u,e_\lambda)\|_{0}\leq
		\varepsilon^{-2}
		\bigl(\|\boldsymbol{\sigma}_h\|_0
		+\|\boldsymbol{\sigma}\|_0\bigr).
	\end{equation}
	Let \(\tilde e_u:=Q_{k-2,h}e_u+(I-Q_{k-2,h})e_u^{\mathrm{CR}}\). Taking
	\(\boldsymbol{\tau}_h=\boldsymbol{\sigma}_h\) in
	\eqref{eq:robust-first-equation} and
	\((v_h,\mu_h)=(e_u,e_\lambda)\) in
	\eqref{eq:fourth-order-discrete-mixed-b}, and using
	\(u_h=I_h^{\mathrm{VE}}u-e_u\) together with
\eqref{eq:vem-hdiv-commuting}, we obtain
	\begin{equation}\label{eq:robust-energy}
		\varepsilon^{-2}\|\boldsymbol{\sigma}_h\|_0^2
		+\|Q_{k-1,h}\nabla_h e_u\|_0^2
		=
		\varepsilon^{-2}(\boldsymbol{\sigma},\boldsymbol{\sigma}_h)
		+c_h(u,e_u)
		-(f,\tilde e_u).
	\end{equation}
	The first term on the right-hand side is bounded by
	\[
	\varepsilon^{-2}(\boldsymbol{\sigma},\boldsymbol{\sigma}_h)
	\lesssim
	\varepsilon^{1/2}\|f\|_0\,
	\varepsilon^{-1}\|\boldsymbol{\sigma}_h\|_0.
	\]
Furthermore, using \(-\Delta\bar u=f\), the regularity estimate
\eqref{eq:perturbation-regularity}, and the same argument as in the
estimate of \(I_2\) in Lemma~\ref{lem:Compatibilityerror}, we get
	\[
	\begin{aligned}
		c_h(u,e_u)-(f,\tilde e_u)
		&=
		c_h(u-\bar u,e_u)
		+c_h(\bar u,e_u)
		+(\Delta\bar u,\tilde e_u)  \\
		&\lesssim
		\bigl(\varepsilon^{1/2}\|f\|_0+h^k\|f\|_{k-1}\bigr)
		|e_u|_{1,h}.
	\end{aligned}
	\]
	Substituting the last two estimates into \eqref{eq:robust-energy} and using the norm equivalence
	\eqref{eq:vem-grad-projection-equivalence}, we obtain
	\[
	\varepsilon^{-1}\|\boldsymbol{\sigma}_h\|_0
	+|e_u|_{1,h}
	\lesssim
	\varepsilon^{1/2}\|f\|_0+h^k\|f\|_{k-1}.
	\]
	Together with \eqref{eq:robust-H2-control} and
	\eqref{eq:robust-pre-est}, this gives
	\[
	\varepsilon|\!|\!|(e_u,e_\lambda)|\!|\!|_{2,h}
	\lesssim
	\varepsilon^{1/2}\|f\|_0+h^k\|f\|_{k-1}.
	\]
	Hence
	\[
	\varepsilon^{-1}\|\boldsymbol{\sigma}_h\|_0
	+
	|\!|\!|(e_u,e_\lambda)|\!|\!|_{\varepsilon,h}
	\lesssim
	\varepsilon^{1/2}\|f\|_0+h^k\|f\|_{k-1}.
	\]
	Therefore, \eqref{eq:robusterror1} follows from the triangle inequality and
	\eqref{eq:robust-pre-est}.
	
	It remains to prove \eqref{eq:robusterror2}. Since
	\(\bar u-u_h=e_u+\bar u-I_h^{\mathrm{VE}}u\), the triangle inequality and \eqref{eq:vem-hdiv-commuting} give
	\[
	\begin{aligned}
		&\varepsilon
		\left|Q_{k-1,h}\nabla_h(\bar u-u_h)\right|_{1,h}
		+|\bar u-u_h|_{1,h} \\
		&\quad\lesssim
		|\!|\!|(e_u,e_\lambda)|\!|\!|_{\varepsilon,h}
		+\varepsilon
		\left|Q_{k-1,h}\nabla(\bar u-u)\right|_{1,h}
		+|\bar u-I_h^{\mathrm{VE}}u|_{1,h}.
	\end{aligned}
	\]
	Using the stability of \(Q_{k-1,h}\), the interpolation estimate \eqref{eq:intererrorVE} and the regularity assumptions
	\eqref{eq:regularity-assumption}-\eqref{eq:perturbation-regularity}, we obtain
	\[
	\begin{aligned}
		&\varepsilon
		\left|Q_{k-1,h}\nabla_h(\bar u-u_h)\right|_{1,h}
		+|\bar u-u_h|_{1,h} \\
		&\quad\lesssim
		|\!|\!|(e_u,e_\lambda)|\!|\!|_{\varepsilon,h}
		+\varepsilon\bigl(|u|_2+|\bar u|_2\bigr)
		+|u-\bar u|_1+h^k|\bar u|_{k+1} \\
		&\quad\lesssim
		\varepsilon^{1/2}\|f\|_0+h^k\|f\|_{k-1}.
	\end{aligned}
	\]
This proves \eqref{eq:robusterror2} and completes the proof.
\end{proof}

\begin{remark}\rm
When \(\varepsilon\eqsim 1\), the estimate
$
|\!|\!|(I_h^{\mathrm{VE}}u-u_h,
Q_{k,\mathcal E_h}\lambda-\lambda_h)|\!|\!|_{\varepsilon,h}
=O(h^{k+1})
$
in \eqref{eq:errorestimate} is superconvergent, whereas
$
\|\boldsymbol\sigma-\boldsymbol\sigma_h\|_0=O(h^{k+1})
$
is of optimal order. In the singularly perturbed regime
\(\varepsilon\to0\), the estimates
\eqref{eq:robusterror1}--\eqref{eq:robusterror2} are both optimal and
robust with respect to \(\varepsilon\).
\end{remark}


\section{Hybridization and Equivalent Formulations}\label{sec:hybridization}
This section hybridizes the normal--normal continuity of the tensor variable in \eqref{eq:fourth-order-discrete-mixed}. After local elimination of the stress variable, the hybridized method yields a stabilization-free weak Galerkin formulation. We then prove that this weak Galerkin formulation is equivalent to an \(H^2\)-nonconforming virtual element method.

\subsection{Hybridized normal--normal continuity}
To hybridize the normal--normal continuity of the stress variable, we
introduce the multiplier space \(\mathbb P_k(\mathring{\mathcal F}_h)\)
and seek the stress variable in the broken space
\[
\Sigma_{k,h}^{-1}:=\mathbb P_k(\mathcal T_h;\mathbb S).
\]
The multiplier imposes normal--normal continuity weakly.
The hybridized mixed formulation of
\eqref{eq:fourth-order-discrete-mixed} reads as follows: find
$
(\boldsymbol\sigma_h,u_h,\gamma_h,\lambda_h)
\in
\Sigma_{k,h}^{-1}\times
\mathring V_{k,h}^{\mathrm{VE}}\times
\mathbb P_k(\mathring{\mathcal F}_h)\times
\mathbb P_k(\mathring{\mathcal E}_h)
$
such that
\begin{subequations}\label{eq:hybridizationmix}
\begin{align}
a(\boldsymbol\sigma_h,\boldsymbol\tau_h)
+
b_h^{\mathrm{hyb}}(\boldsymbol\tau_h;u_h,\gamma_h,\lambda_h)
&=0,
\label{eq:hybridizationmix1}\\
b_h^{\mathrm{hyb}}(\boldsymbol\sigma_h;v_h,\chi_h,\mu_h)
-
c_h(u_h,v_h)
&=-(f,\widetilde v_h),
\label{eq:hybridizationmix2}
\end{align}
\end{subequations}
for all
$
(\boldsymbol\tau_h,v_h,\chi_h,\mu_h)
\in
\Sigma_{k,h}^{-1}\times
\mathring V_{k,h}^{\mathrm{VE}}\times
\mathbb P_k(\mathring{\mathcal F}_h)\times
\mathbb P_k(\mathring{\mathcal E}_h),
$
where
\begin{align*}
b_h^{\mathrm{hyb}}(\boldsymbol{\tau}_h; v_h,\chi_h,\mu_h)
 &:=
 -\sum_{T\in\mathcal T_h}\Bigl[
 (\div\div\boldsymbol{\tau}_h, Q_{k-2,T} v_h)_T +(\bs n_F^{\intercal}\boldsymbol{\tau}_h\bs n_{\partial T}, \chi_h)_{\partial T}
 \\
 &\qquad\quad
  -(\tr_2(\boldsymbol{\tau}_h), Q_{k-1,F}v_h)_{\partial T}+\sum_{e\in\Delta_{d-2}(T)}(\tr_e(\boldsymbol{\tau}_h), \mu_h)_e
 \Bigr].
 \end{align*}
If \(\boldsymbol\tau_h\in\Sigma_{k,h}^{nn}\), then
\begin{equation}\label{eq:hyb-reduction}
b_h^{\mathrm{hyb}}(\boldsymbol{\tau}_h; v_h,\chi_h,\mu_h)=b_h(\boldsymbol{\tau}_h;v_h,\mu_h)\;\;\forall\, v_h \in \mathring{V}_{k,h}^{\mathrm{VE}}, \chi_h\in\mathbb P_k(\mathring{\mathcal F}_h), \mu_h \in \mathbb P_k(\mathring{\mathcal E}_h).
\end{equation}

\begin{theorem}\label{thm:hybridized-equivalence}
The hybridized mixed method \eqref{eq:hybridizationmix} is well posed.
Moreover, it is equivalent to the distributional mixed method
\eqref{eq:fourth-order-discrete-mixed} in the following sense: if
$
(\boldsymbol\sigma_h,u_h,\gamma_h,\lambda_h)
\in
\Sigma_{k,h}^{-1}\times
\mathring V_{k,h}^{\mathrm{VE}}\times
\mathbb P_k(\mathring{\mathcal F}_h)\times
\mathbb P_k(\mathring{\mathcal E}_h)
$
solves \eqref{eq:hybridizationmix}, then
$
(\boldsymbol\sigma_h,u_h,\lambda_h)
\in
\Sigma_{k,h}^{nn}\times
\mathring V_{k,h}^{\mathrm{VE}}\times
\mathbb P_k(\mathring{\mathcal E}_h)
$
solves \eqref{eq:fourth-order-discrete-mixed}.
\end{theorem}
\begin{proof}
We first prove well-posedness. It suffices to show that the homogeneous
problem has only the trivial solution. Let \(f=0\). Arguing as in the
proof of Theorem~\ref{thm:fourth-orderdistmixwellposedness}, we obtain
\(\boldsymbol\sigma_h=0\) and \(u_h=0\). Hence
\eqref{eq:hybridizationmix1} reduces to
\begin{equation}\label{eq:20260510}
b_h^{\mathrm{hyb}}(\boldsymbol\tau_h;0,\gamma_h,\lambda_h)=0
\qquad
\forall\,\boldsymbol\tau_h\in\Sigma_{k,h}^{-1}.
\end{equation}
In particular, by restricting the test functions to
\(\Sigma_{k,h}^{nn}\) and using \eqref{eq:hyb-reduction}, we have
\[
b_h(\boldsymbol\tau_h;0,\lambda_h)=0
\qquad
\forall\,\boldsymbol\tau_h\in\Sigma_{k,h}^{nn}.
\]
The proof of Theorem~\ref{thm:fourth-orderdistmixwellposedness} then gives
\(\lambda_h=0\). Therefore \eqref{eq:20260510} becomes
\[
\sum_{T\in\mathcal T_h}
(\boldsymbol n_F^{\intercal}\boldsymbol\tau_h\boldsymbol n_{\partial T},
\gamma_h)_{\partial T}
=0
\qquad
\forall\,\boldsymbol\tau_h\in\Sigma_{k,h}^{-1}.
\]
Since the stress space is broken, the face degrees of freedom
\eqref{eq:nnDoFs1} allow us to choose \(\boldsymbol\tau_h\) locally so that
$
\boldsymbol n_F^{\intercal}\boldsymbol\tau_h\boldsymbol n_{\partial T}
=\gamma_h
$
on each face. It follows that \(\gamma_h=0\). Thus the homogeneous problem
has only the zero solution, and the finite-dimensional system is well
posed.

We next prove the equivalence. Let
\((\boldsymbol\sigma_h,u_h,\gamma_h,\lambda_h)\) be the solution of
\eqref{eq:hybridizationmix}. Taking \(v_h=0\) and \(\mu_h=0\) in
\eqref{eq:hybridizationmix2}, and using arbitrary
\(\chi_h\in\mathbb P_k(\mathring{\mathcal F}_h)\), we obtain
\(\boldsymbol\sigma_h\in\Sigma_{k,h}^{nn}\). Therefore, by
\eqref{eq:hyb-reduction}, \eqref{eq:hybridizationmix2} reduces to
\eqref{eq:fourth-order-discrete-mixed-b}. Restricting
\(\boldsymbol\tau_h\) to \(\Sigma_{k,h}^{nn}\) in
\eqref{eq:hybridizationmix1}, and again using \eqref{eq:hyb-reduction}, we
obtain \eqref{eq:fourth-order-discrete-mixed-a}. Thus
\((\boldsymbol\sigma_h,u_h,\lambda_h)\) solves
\eqref{eq:fourth-order-discrete-mixed}.
\end{proof}

\subsection{Stabilization-free weak Galerkin method}
For
$
(v,\chi,\mu)\in
\mathring H_1^{\mathrm{nc}}(\mathcal T_h)
\times L^2(\mathring{\mathcal F}_h)
\times L^2(\mathring{\mathcal E}_h)$,
we define the hybrid weak Hessian
\(\nabla_{w,\mathrm{hyb}}^2(v,\chi,\mu)\in\Sigma_{k,h}^{-1}\) elementwise
by
\[
\begin{aligned}
(\nabla_{w,\mathrm{hyb}}^2(v,\chi,\mu),\boldsymbol\tau)_T
:={}&
(Q_{k-2,T}v,\operatorname{div}\operatorname{div}\boldsymbol\tau)_T
+
(\chi,\boldsymbol n_F^{\intercal}\boldsymbol\tau
\boldsymbol n_{\partial T})_{\partial T}
\\
&-
(Q_{k-1,F}v,\operatorname{tr}_2(\boldsymbol\tau))_{\partial T}
+
\sum_{e\in\Delta_{d-2}(T)}
(\mu,\operatorname{tr}_e(\boldsymbol\tau))_e
\end{aligned}
\]
for all \(\boldsymbol\tau\in\mathbb P_k(T;\mathbb S)\) and
\(T\in\mathcal T_h\).

By definition, for $\boldsymbol{\tau} \in \Sigma_{k,h}^{-1}$, $v \in \mathring H_1^{\mathrm{nc}}(\mathcal T_h)$, $\chi\in L^2(\mathring{\mathcal F}_h)$ and $\mu \in L^2(\mathring{\mathcal E}_h)$ 
\[
b_h^{\mathrm{hyb}}(\boldsymbol\tau;v,\chi,\mu)
=
-(\boldsymbol\tau,\nabla_{w,\mathrm{hyb}}^2(v,\chi,\mu)).
\]
Consequently, the first equation \eqref{eq:hybridizationmix1} is equivalent
to
\[
\boldsymbol\sigma_h
=
\varepsilon^2
\nabla_{w,\mathrm{hyb}}^2(u_h,\gamma_h,\lambda_h).
\]
Substituting this identity into \eqref{eq:hybridizationmix2}, we obtain the
following stabilization-free weak Galerkin formulation: find
$
(u_h,\gamma_h,\lambda_h)\in\mathring M_{k,h}^{\mathrm{hyb}}
$
such that
\begin{equation}\label{eq:WG}
\varepsilon^2
(\nabla_{w,\mathrm{hyb}}^2(u_h,\gamma_h,\lambda_h),
\nabla_{w,\mathrm{hyb}}^2(v,\chi,\mu))
+
(Q_{k-1,h}\nabla_h u_h,Q_{k-1,h}\nabla_h v)
=
(f,\widetilde v)
\end{equation}
for all \((v,\chi,\mu)\in\mathring M_{k,h}^{\mathrm{hyb}}\), where
\[
\mathring M_{k,h}^{\mathrm{hyb}}
:=
\mathring V_{k,h}^{\mathrm{VE}}
\times
\mathbb P_k(\mathring{\mathcal F}_h)
\times
\mathbb P_k(\mathring{\mathcal E}_h).
\]

\subsection{Stabilization-free \(H^2\)-nonconforming virtual element formulation}
We recall the \(H^2\)-nonconforming virtual element space from
\cite{ChenHuang2020nonconforming}. For \(k\ge1\), let
\begin{align*}
\mathring{W}_{k+2,h}^{\mathrm{VE}}
:= \{u\in L^2(\Omega):\ &
u|_T\in W_{k+2}^{\mathrm{VE}}(T)\; \textrm{ for } T\in\mathcal T_h;\ \textrm{ all the DoFs in \eqref{eq:H2NCVEM-DOFs} are } 
 \\
& \text{single-valued on } \mathring{\mathcal F}_h \textrm{ and } 
\mathring{\mathcal E}_h,\ \text{and vanish on }\partial\Omega\},
\end{align*}
where 
the local shape function space is
   \begin{align*}
    W_{k+2}^{\mathrm{VE}}(T):=&\{u\in H^2(T):
    \Delta^2u\in\mathbb P_{k-2}(T);\ \tr_e(\nabla^2u)\in\mathbb P_k(e)\, \textrm{ for } e\in\Delta_{d-2}(T);
     \\
    &\quad\; \tr_1(\nabla^2u)|_F\in\mathbb P_k(F),\, \tr_2(\nabla^2u)|_F\in\mathbb P_{k-1}(F) 
    \, \textrm{ for } F\in\Delta_{d-1}(T)\}.
   \end{align*}
A unisolvent set of DoFs is given by
\begin{equation}\label{eq:H2NCVEM-DOFs}
(Q_{k-2,T}u,\; Q_{k-1,F}u,\; Q_{k,F}(\partial_{n_F}u),\; Q_{k,e}u).
\end{equation}

   Since the first two groups of DoFs in \eqref{eq:H2NCVEM-DOFs}
   coincide with those of \(\mathring V_{k,h}^{\mathrm{VE}}\),  the DoFs in \eqref{eq:H2NCVEM-DOFs} induce the natural isomorphism $Q_M:\mathring{W}_{k+2,h}^{\mathrm{VE}}\to\mathring M_{k,h}^{\mathrm{hyb}}$ defined by 
   \[
   Q_Mv_h:=
   \bigl(I_h^{\mathrm{VE}}v_h,\ Q_{k,F}(\partial_{n_F}v_h),\ Q_{k,e}v_h\bigr),\quad F\in\mathring{\mathcal F}_h,\,e\in\mathring{\mathcal E}_h.
   \]
   
\begin{lemma}\label{lem:hyb-weak-Hessian-VEM}
For \(k\ge1\), it holds that
\begin{equation}\label{eq:operator-identities}
\nabla_{w,\mathrm{hyb}}^2(Q_M v_h)
=
Q_{k,h}\nabla_h^2 v_h
\qquad
\forall\,v_h\in\mathring{W}_{k+2,h}^{\mathrm{VE}}.
\end{equation}
\end{lemma}

\begin{proof}
It follows from the Green identity
\eqref{eq:greenidentitydivdiv} and the definitions of
\(\nabla_{w,\mathrm{hyb}}^2\) and \(Q_M\).
\end{proof}

Under the isomorphism \(Q_M\), the weak Galerkin scheme \eqref{eq:WG} can
be rewritten as the following stabilization-free \(H^2\)-nonconforming
virtual element method: find \(u_h\in\mathring{W}_{k+2,h}^{\mathrm{VE}}\)
such that
\begin{equation}\label{eq:VEMWG}
\varepsilon^2
(Q_{k,h}\nabla_h^2 u_h,Q_{k,h}\nabla_h^2 v_h)
+
(Q_{k-1,h}\nabla_h I_h^{\mathrm{VE}}u_h,
Q_{k-1,h}\nabla_h I_h^{\mathrm{VE}}v_h)
=
(f,\widetilde v_h)
\end{equation}
for all \(v_h\in\mathring{W}_{k+2,h}^{\mathrm{VE}}\). By the norm
equivalence in \cite[(4.14)]{ChenHuang2025},
\(\|Q_{k,h}\nabla_h^2 v_h\|_0\) is a norm on
\(\mathring{W}_{k+2,h}^{\mathrm{VE}}\). Hence
\eqref{eq:VEMWG} is well posed.

\begin{theorem}\label{thm:VEM-WG-equivalence}
The \(H^2\)-nonconforming virtual element formulation \eqref{eq:VEMWG} is
equivalent to the weak Galerkin formulation \eqref{eq:WG}, and hence to
the distributional mixed method \eqref{eq:fourth-order-discrete-mixed}.
More precisely, if \(u_h\in\mathring{W}_{k+2,h}^{\mathrm{VE}}\) solves
\eqref{eq:VEMWG}, then \(Q_Mu_h\in\mathring M_{k,h}^{\mathrm{hyb}}\)
solves \eqref{eq:WG}; conversely, every solution of \eqref{eq:WG} is the
image under \(Q_M\) of a unique solution of \eqref{eq:VEMWG}.
\end{theorem}

\begin{proof}
The assertion follows from the operator identity
\eqref{eq:operator-identities} and the fact that \(Q_M\) is an isomorphism
between \(\mathring{W}_{k+2,h}^{\mathrm{VE}}\) and
\(\mathring M_{k,h}^{\mathrm{hyb}}\).
\end{proof}  
  
 
 \section{Numerical Results}\label{sec:numerresults}
 
We present two numerical tests for the lowest-order case \(k=1\). The smooth
manufactured-solution test checks the estimates \eqref{eq:robusterror1},
\eqref{eq:errorestimate}, and \eqref{eq:robusterror2}. The L-shaped constant-load benchmark is a non-manufactured
clamped-plate test on a nonsmooth polygonal domain. It examines how the adaptive
method captures mesh concentration near the reentrant corner and, for small
\(\varepsilon\), along the clamped boundary. All computations are carried out in MATLAB using
iFEM~\cite{Chen2009}.

\subsection{Smooth Parameter-Uniform Tests}

On the unit cube \(\Omega=(0,1)^3\), uniformly refined simplicial meshes are
used. For
\[
u=\sin^2(\pi x)\sin^2(\pi y)\sin^2(\pi z),
\]
the load \(f\) is obtained from \eqref{eq:fourthorderequation}. We measure
\[
\mathrm{Err}_1
:=
\varepsilon^{-1}\|\boldsymbol{\sigma}-\boldsymbol{\sigma}_h\|_{0}
+
|\!|\!|(u-u_h,\lambda-\lambda_h)|\!|\!|_{\varepsilon,h}.
\]
For the optimal estimate, we also compute
\[
\mathrm{Err}_u
:=
|\!|\!|(I_h^{\mathrm{VE}}u-u_h,
Q_{k,\mathcal{E}_h}\lambda-\lambda_h)|\!|\!|_{\varepsilon,h},
\qquad
\mathrm{Err}_\sigma
:=
\varepsilon^{-1}\|\boldsymbol{\sigma}-\boldsymbol{\sigma}_h\|_{0}.
\]
Table~\ref{tab:robust3d} reports the robust error for
\(\varepsilon=1,10^{-1},10^{-5},10^{-6}\), and
Table~\ref{tab:3D_sigma_error} reports the optimal components for
\(\varepsilon=1,10^{-1}\).

\begin{table}[htbp]
	\centering
	\caption{Robust error \(\mathrm{Err}_1\) on the smooth three-dimensional manufactured solution.}
	\label{tab:robust3d}
	\small
	\setlength{\tabcolsep}{4pt}
	\renewcommand{\arraystretch}{1.12}
	\begin{tabular}{cc c c c c c c c c}
		\toprule
		\multirow{2}{*}{\(k\)} & \multirow{2}{*}{\(h\)}
		& \multicolumn{2}{c}{\(\varepsilon = 1\)}
		& \multicolumn{2}{c}{\(\varepsilon = 10^{-1}\)}
		& \multicolumn{2}{c}{\(\varepsilon = 10^{-5}\)}
		& \multicolumn{2}{c}{\(\varepsilon = 10^{-6}\)} \\
		\cmidrule(lr){3-4}\cmidrule(lr){5-6}\cmidrule(lr){7-8}\cmidrule(lr){9-10}
		& & \(\mathrm{Err}_1\) & Rate & \(\mathrm{Err}_1\) & Rate
		& \(\mathrm{Err}_1\) & Rate & \(\mathrm{Err}_1\) & Rate \\
		\midrule
		\multirow{5}{*}{\(1\)}
		& \(1/2\)  & 7.299e+00 & --   & 1.449e+00 & --   & 9.836e-01 & --   & 9.835e-01 & --   \\
		& \(1/4\)  & 2.524e+00 & 1.53 & 7.950e-01 & 0.87 & 7.023e-01 & 0.49 & 7.022e-01 & 0.49 \\
		& \(1/8\)  & 8.433e-01 & 1.58 & 3.616e-01 & 1.14 & 3.639e-01 & 0.95 & 3.637e-01 & 0.95 \\
		& \(1/16\) & 2.898e-01 & 1.54 & 1.665e-01 & 1.12 & 1.836e-01 & 0.99 & 1.835e-01 & 0.99 \\
		& \(1/32\) & 1.099e-01 & 1.40 & 7.910e-02 & 1.07 & 9.208e-02 & 1.00 & 9.195e-02 & 1.00 \\
		\bottomrule
	\end{tabular}
\end{table}

\begin{table}[htbp]
	\centering
	\caption{Optimal components \(\mathrm{Err}_u\) and \(\mathrm{Err}_\sigma\) on the smooth manufactured solution.}
	\label{tab:3D_sigma_error}
	\small
	\setlength{\tabcolsep}{4pt}
	\renewcommand{\arraystretch}{1.12}
	\begin{tabular}{cc c c c c c c c c}
		\toprule
		\multirow{2}{*}{\(k\)} & \multirow{2}{*}{\(h\)}
		& \multicolumn{4}{c}{\(\varepsilon = 1\)}
		& \multicolumn{4}{c}{\(\varepsilon = 10^{-1}\)} \\
		\cmidrule(lr){3-6}\cmidrule(lr){7-10}
		& & \(\mathrm{Err}_\sigma\) & Rate & \(\mathrm{Err}_u\) & Rate
		& \(\mathrm{Err}_\sigma\) & Rate & \(\mathrm{Err}_u\) & Rate \\
		\midrule
		\multirow{5}{*}{\(1\)}
		& \(1/2\)   & 5.896e+00 & --    & 1.105e+00 & --    & 5.726e-01 & --    & 1.415e-01 & --    \\
		& \(1/4\)   & 1.857e+00 & 1.67 & 3.566e-01 & 1.63 & 2.258e-01 & 1.34 & 7.590e-02 & 0.90 \\
		& \(1/8\)   & 5.245e-01 & 1.82 & 1.224e-01 & 1.54 & 6.653e-02 & 1.76 & 2.104e-02 & 1.85 \\
		& \(1/16\)  & 1.372e-01 & 1.94 & 3.409e-02 & 1.84 & 1.763e-02 & 1.92 & 5.395e-03 & 1.96 \\
		& \(1/32\)  & 3.483e-02 & 1.98 & 8.823e-03 & 1.95 & 4.499e-03 & 1.97 & 1.360e-03 & 1.99 \\
		\bottomrule
	\end{tabular}
\end{table}
 
For the reduced-limit estimate \eqref{eq:robusterror2}, we take
\(\bar u=\sin(\pi x)\sin(\pi y)\sin(\pi z)\) in \eqref{eq:poisson} and compute
\[
\mathrm{Err}_2
:=
\varepsilon \Bigl( \sum_{T \in \mathcal{T}_h} h_T^{-4}
\big\| Q_{k-2,T}(\bar u-u_h-\bar u^{\mathrm{CR}}+u_h^{\mathrm{CR}})
\big\|_{0,T}^2 \Bigr)^{1/2}
+
\left\|\nabla \bar u-Q_{k-1,h}\nabla_h u_h \right\|_0 .
\]
Table~\ref{tab:3D_Qhdiv_error} reports the results for
\(\varepsilon=10^{-6},10^{-8},10^{-10}\). The rates in
Tables~\ref{tab:robust3d}--\ref{tab:3D_Qhdiv_error} are consistent with
\eqref{eq:robusterror1}, \eqref{eq:errorestimate}, and
\eqref{eq:robusterror2}, respectively, with no visible deterioration as
\(\varepsilon\) decreases.

\begin{table}[htbp]
	\centering
	\caption{Reduced-limit error \(\mathrm{Err}_2\) on the smooth three-dimensional test.}
	\label{tab:3D_Qhdiv_error}
	\small
	\setlength{\tabcolsep}{4pt}
	\renewcommand{\arraystretch}{1.12}
	\begin{tabular}{cc c c c c c c}
		\toprule
		\multirow{2}{*}{$k$} & \multirow{2}{*}{$h$}
		& \multicolumn{2}{c}{$\varepsilon = 10^{-6}$}
		& \multicolumn{2}{c}{$\varepsilon = 10^{-8}$}
		& \multicolumn{2}{c}{$\varepsilon = 10^{-10}$} \\
		\cmidrule(lr){3-4}\cmidrule(lr){5-6}\cmidrule(lr){7-8}
		& & $\mathrm{Err}_2$ & Rate & $\mathrm{Err}_2$ & Rate & $\mathrm{Err}_2$ & Rate \\
		\midrule
		\multirow{5}{*}{$1$}
		& $1/2$   & 1.106e+00 & --    & 1.106e+00 & --    & 1.106e+00 & --    \\
		& $1/4$   & 5.815e-01 & 0.93 & 5.815e-01 & 0.93 & 5.815e-01 & 0.93 \\
		& $1/8$   & 2.942e-01 & 0.98 & 2.942e-01 & 0.98 & 2.942e-01 & 0.98 \\
		& $1/16$  & 1.475e-01 & 1.00 & 1.475e-01 & 1.00 & 1.475e-01 & 1.00 \\
		& $1/32$  & 7.381e-02 & 1.00 & 7.381e-02 & 1.00 & 7.381e-02 & 1.00 \\
		\bottomrule
	\end{tabular}
\end{table}

\subsection{Adaptive L-Shaped Constant-Load Test}
\label{sec:k1-adaptive-lshape-load}

We consider a clamped plate under the uniform load \(f=1\) on the L-shaped
domain \(\Omega_L=(-1,1)^2\setminus([0,1]\times[-1,0])\), whose reentrant
corner is the origin. Since no exact solution is available, the adaptive
computation is used to guide mesh refinement and to visualize the recovered
curvature near the corner singularity and, for small \(\varepsilon\), along
boundary layers. The
model problem is
\[
\varepsilon^2\Delta^2u-\Delta u=1\quad\hbox{in }\Omega_L,
\qquad
u=\partial_nu=0\quad\hbox{on }\partial\Omega_L .
\]
In this adaptive test, the element size used in the indicator and in the
mesh-density plots is the area scale \(h_T=|T|^{1/2}\).
For \(T\in\mathcal T_h\), define
\[
\kappa_T=\min\{1,h_T/\varepsilon\}.
\]
The adaptive indicator used for marking is
\[
\eta_h^2
=\eta_R^2+\eta_t^2+\eta_{\sigma}^2+\eta_\lambda^2,
\]
where
\[
\begin{aligned}
\eta_R^2&=\sum_{T\in\mathcal T_h}(\kappa_T h_T)^2
\|Q_{0,T}f\|_{0,T}^2,\qquad
\eta_t^2=\sum_{F\in\mathcal F_h}h_F
\|[\![\nabla_hu_h\cdot t_F]\!]\|_{0,F}^2,\\
\eta_{\sigma}^2
&:=
\varepsilon^2\sum_{T\in\mathcal T_h}h_T^2
\|\operatorname{rot}(\varepsilon^{-2}\boldsymbol\sigma_h)\|_{0,T}^2 +
\varepsilon^2\sum_{F\in\mathcal F_h}h_F
\|[\![(\varepsilon^{-2}\boldsymbol\sigma_h)t_F]\!]\|_{0,F}^2,\\
\eta_\lambda^2&=\varepsilon^2
\sum_{T\in\mathcal T_h}\sum_{a\in\Delta_0(T)}
h_T^{-2}\bigl|u_h|_T(a)-\lambda_h(a)\bigr|^2 .
\end{aligned}
\]
For a piecewise smooth symmetric tensor \(\boldsymbol\tau\), set
\[
\operatorname{rot}\boldsymbol\tau
:=
\begin{pmatrix}
\partial_1\tau_{12}-\partial_2\tau_{11}\\
\partial_1\tau_{22}-\partial_2\tau_{12}
\end{pmatrix}.
\]
Here \(Q_{0,T}f=f\), so the data oscillation is zero. We use D\"orfler marking
with \(\theta=0.5\).

Fig.~\ref{fig:k1-lshape-mesh-curvature} compares adaptive results at
comparable numbers of degrees of freedom. The first two panels show the mesh
density \(\log_2(h_{\max}/h_T)\), where larger values indicate finer elements.
The third panel shows a smoothed logarithmic visualization of
\(\log_{10}(1+\|\varepsilon^{-2}\boldsymbol\sigma_h\|_F)\) for
\(\varepsilon=10^{-3}\) on the mesh in the middle panel. Compared with \(\varepsilon=1\), the
\(\varepsilon=10^{-3}\) mesh shows sharper concentration along the clamped
boundary. In the finest refined boundary regions, \(h_T\) is smaller than
\(\varepsilon\), indicating that the adaptive mesh resolves the expected
\(O(\varepsilon)\) boundary-layer scale.

\begin{figure}[htbp]
	\centering
	\begin{minipage}[b]{0.31\textwidth}
		\centering
		\includegraphics[width=\textwidth,trim=0pt 0pt 45pt 18pt,clip]{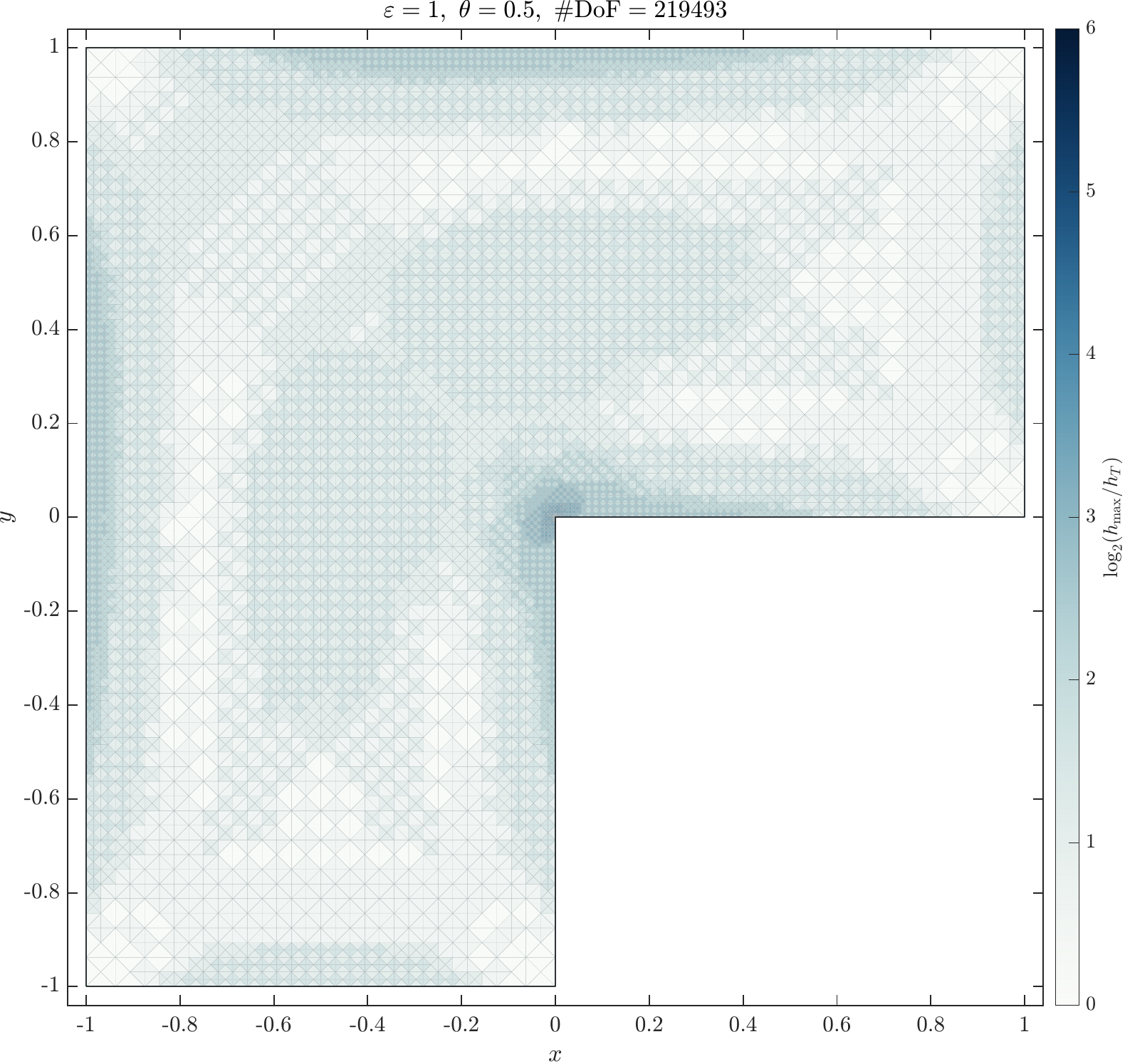}\\[-1mm]
		{\small (a) Mesh density, \(\varepsilon=1\).}
	\end{minipage}\hfill
	\begin{minipage}[b]{0.31\textwidth}
		\centering
		\includegraphics[width=\textwidth,trim=0pt 0pt 45pt 18pt,clip]{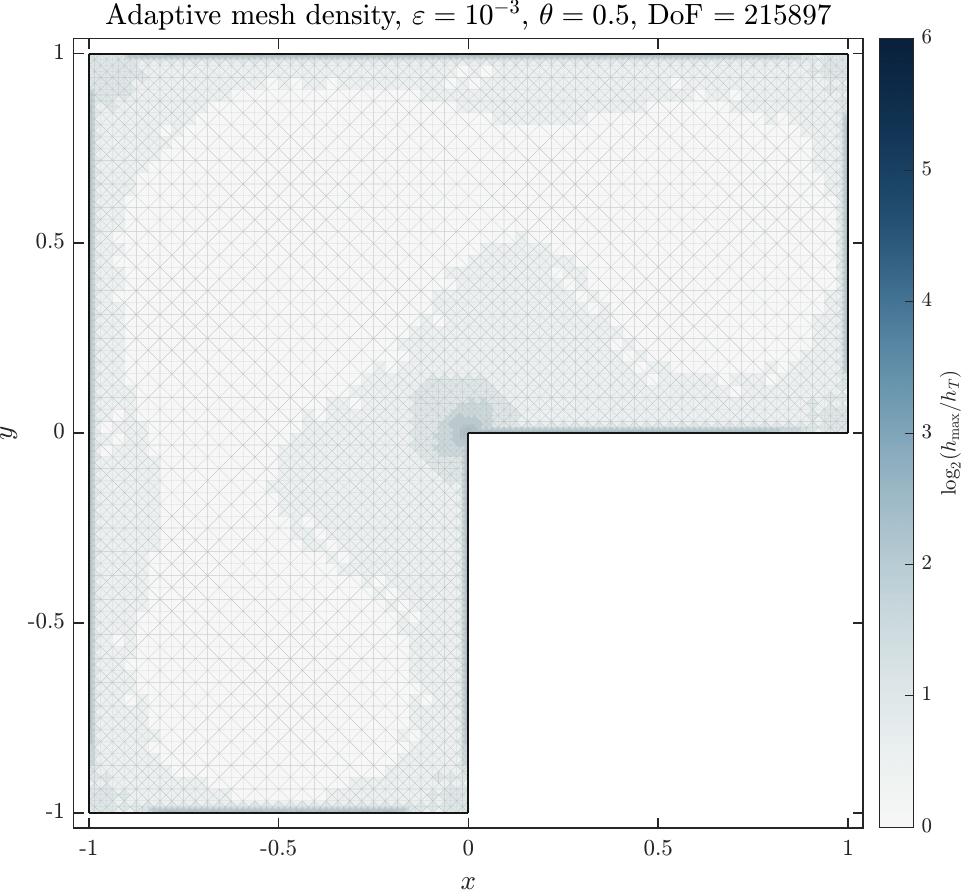}\\[-1mm]
		{\small (b) Mesh density, \(\varepsilon=10^{-3}\).}
	\end{minipage}\hfill
	\begin{minipage}[b]{0.34\textwidth}
		\centering
		\includegraphics[width=\textwidth,trim=30pt 0pt 155pt 85pt,clip]{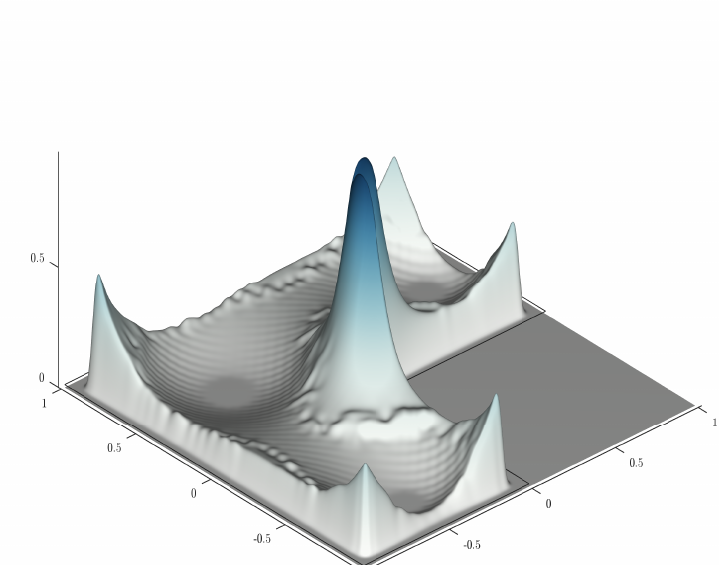}\\[-1mm]
		{\small (c) Recovered-curvature surface.}
	\end{minipage}
	\caption{Adaptive results for the L-shaped test with \(f=1\) and
		\(\theta=0.5\): (a) mesh density for \(\varepsilon=1\) with
		\(219493\) free degrees of freedom; (b) mesh density for
		\(\varepsilon=10^{-3}\) with \(215897\) free degrees of freedom; and
		(c)
		smoothed \(\log_{10}(1+\|\varepsilon^{-2}\boldsymbol\sigma_h\|_F)\)
		surface for \(\varepsilon=10^{-3}\). Here \(\|\cdot\|_F\) denotes the Frobenius norm.}
	\label{fig:k1-lshape-mesh-curvature}
\end{figure}

\bibliographystyle{siamplain}
\bibliography{references_nodoi}

\begin{thebibliography}{10}

\bibitem{Adolfsson1992}
{\sc V.~Adolfsson}, {\em {$L^2$}-integrability of second order derivatives for
  {Poisson}'s equation in nonsmooth domains}, Math. Scand., 70 (1992),
  pp.~146--160.

\bibitem{De2016nonconforming}
{\sc B.~Ayuso~de Dios, K.~Lipnikov, and G.~Manzini}, {\em The nonconforming
  virtual element method}, ESAIM Math. Model. Numer. Anal., 50 (2016),
  pp.~879--904.

\bibitem{BehrensGuzman2011}
{\sc E.~M. Behrens and J.~Guzm\'an}, {\em A mixed method for the biharmonic
  problem based on a system of first-order equations}, SIAM J. Numer. Anal., 49
  (2011), pp.~789--817.

\bibitem{BerchenkoKoganGawlik2025}
{\sc Y.~Berchenko-Kogan and E.~S. Gawlik}, {\em Finite element spaces of double
  forms}, arXiv preprint arXiv:2505.17243,  (2025).

\bibitem{BoffiBrezziFortin2013}
{\sc D.~Boffi, F.~Brezzi, and M.~Fortin}, {\em Mixed finite element methods and
  applications}, vol.~44 of Springer Series in Computational Mathematics,
  Springer, Heidelberg, 2013.

\bibitem{Brenner2003}
{\sc S.~C. Brenner}, {\em Poincar\'{e}-{F}riedrichs inequalities for piecewise
  {$H^1$} functions}, SIAM J. Numer. Anal., 41 (2003), pp.~306--324.

\bibitem{BrennerNeilan2011}
{\sc S.~C. Brenner and M.~Neilan}, {\em A {$C^0$} interior penalty method for a
  fourth order elliptic singular perturbation problem}, SIAM J. Numer. Anal.,
  49 (2011), pp.~869--892.

\bibitem{Brezzi1987}
{\sc F.~Brezzi, J.~Douglas, Jr., R.~Dur\'{a}n, and M.~Fortin}, {\em Mixed
  finite elements for second order elliptic problems in three variables},
  Numer. Math., 51 (1987), pp.~237--250.

\bibitem{Brezzi1985}
{\sc F.~Brezzi, J.~Douglas, Jr., and L.~D. Marini}, {\em Two families of mixed
  finite elements for second order elliptic problems}, Numer. Math., 47 (1985),
  pp.~217--235.

\bibitem{CarstensenHeuer2025}
{\sc C.~Carstensen and N.~Heuer}, {\em Normal-normal continuous symmetric
  stresses in mixed finite element elasticity}, Math. Comp., 94 (2025),
  pp.~1571--1602.

\bibitem{CarstensenHeuer2026}
{\sc C.~Carstensen and N.~Heuer}, {\em Normal-normal continuous symmetric
  stress approximation in three-dimensional linear elasticity}, Numer. Math.,
  (2026).

\bibitem{ChenChenHuangWei2024}
{\sc C.~Chen, L.~Chen, X.~Huang, and H.~Wei}, {\em Geometric decomposition and
  efficient implementation of high order face and edge elements}, Commun.
  Comput. Phys., 35 (2024), pp.~1045--1072.

\bibitem{Chen2009}
{\sc L.~Chen}, {\em {iFEM}: an integrated finite element methods package in
  {MATLAB}}, tech. report, University of California, Irvine, 2009.

\bibitem{ChenHuang2020nonconforming}
{\sc L.~Chen and X.~Huang}, {\em Nonconforming virtual element method for
  {$2m$}-th order partial differential equations in {$\mathbb{R}^n$}}, Math.
  Comp., 89 (2020), pp.~1711--1744.

\bibitem{ChenHuang2022}
{\sc L.~Chen and X.~Huang}, {\em Finite elements for {div}- and
  {divdiv}-conforming symmetric tensors in arbitrary dimension}, SIAM J. Numer.
  Anal., 60 (2022), pp.~1932--1961.

\bibitem{ChenHuang2022a}
{\sc L.~Chen and X.~Huang}, {\em Finite elements for {${\rm div\,div}$}
  conforming symmetric tensors in three dimensions}, Math. Comp., 91 (2022),
  pp.~1107--1142.

\bibitem{Chen2024}
{\sc L.~Chen and X.~Huang}, {\em {$H(\operatorname{div})$}-conforming finite
  element tensors with constraints}, Results Appl. Math., 23 (2024), p.~Paper
  No. 100494.

\bibitem{ChenHuang2025}
{\sc L.~Chen and X.~Huang}, {\em A new {div-div-conforming} symmetric tensor
  finite element space with applications to the biharmonic equation}, Math.
  Comp., 94 (2025), pp.~33--72.

\bibitem{Ciarlet1997Plates}
{\sc P.~G. Ciarlet}, {\em Mathematical Elasticity. Volume II: Theory of
  Plates}, vol.~27 of Studies in Mathematics and its Applications,
  North-Holland, Amsterdam, 1997.

\bibitem{CrouzeixRaviart1973}
{\sc M.~Crouzeix and P.-A. Raviart}, {\em Conforming and nonconforming finite
  element methods for solving the stationary {S}tokes equations. {I}}, Rev.
  Fran\c{c}aise Automat. Informat. Recherche Op\'{e}rationnelle S\'{e}r. Rouge,
  7 (1973), pp.~33--75.

\bibitem{CuiHuang2025}
{\sc X.~Cui and X.~Huang}, {\em Low-order finite element complex with
  application to a fourth-order elliptic singular perturbation problem}, SIAM
  J. Numer. Anal., 64 (2026), pp.~828--852.

\bibitem{DongErn2021}
{\sc Z.~Dong and A.~Ern}, {\em Hybrid high-order method for singularly
  perturbed fourth-order problems on curved domains}, ESAIM Math. Model. Numer.
  Anal., 55 (2021), pp.~3091--3114.

\bibitem{FengYu2024}
{\sc F.~Feng and Y.~Yu}, {\em A modified interior penalty virtual element
  method for fourth-order singular perturbation problems}, J. Sci. Comput., 101
  (2024), p.~Paper No. 21.

\bibitem{FranzRoosWachtel2014}
{\sc S.~Franz, H.~G. Roos, and A.~Wachtel}, {\em A {$C^0$} interior penalty
  method for a singularly-perturbed fourth-order elliptic problem on a
  layer-adapted mesh}, Numer. Methods Partial Differential Equations, 30
  (2014), pp.~838--861.

\bibitem{GaoLai2020}
{\sc F.~Gao and M.-J. Lai}, {\em A new {$H^2$} regularity condition of the
  solution to the {Dirichlet} problem for the {Poisson} equation and its
  applications}, Acta Math. Sin. (Engl. Ser.), 36 (2020), pp.~21--39.

\bibitem{GurtinFriedAnand2010}
{\sc M.~E. Gurtin, E.~Fried, and L.~Anand}, {\em The Mechanics and
  Thermodynamics of Continua}, Cambridge University Press, Cambridge, 2010.

\bibitem{GuzmanLeykekhmanNeilan2012}
{\sc J.~Guzm\'{a}n, D.~Leykekhman, and M.~Neilan}, {\em A family of
  non-conforming elements and the analysis of {Nitsche's} method for a
  singularly perturbed fourth order problem}, Calcolo, 49 (2012), pp.~95--125.

\bibitem{Hellan1967}
{\sc K.~Hellan}, {\em Analysis of elastic plates in flexure by a simplified
  finite element method}, Acta Polytech. Scand. Civ. Eng. Build. Constr. Ser.,
  46 (1967).

\bibitem{Herrmann1967}
{\sc L.~R. Herrmann}, {\em Finite-element bending analysis for plates}, J. Eng.
  Mech. Div., 93 (1967), pp.~13--26.

\bibitem{HuLin2025}
{\sc K.~Hu and T.~Lin}, {\em Finite element form-valued forms: Construction},
  arXiv preprint arXiv:2503.03243,  (2025).

\bibitem{HuangShiWang2021}
{\sc X.~Huang, Y.~Shi, and W.~Wang}, {\em A {Morley-Wang-Xu} element method for
  a fourth order elliptic singular perturbation problem}, J. Sci. Comput., 87
  (2021), p.~Paper No. 84.

\bibitem{HuangTang2025}
{\sc X.~Huang and Z.~Tang}, {\em Robust and optimal mixed methods for a
  fourth-order elliptic singular perturbation problem}, J. Sci. Comput., 105
  (2025), p.~Paper No. 72.

\bibitem{Johnson1973}
{\sc C.~Johnson}, {\em On the convergence of a mixed finite-element method for
  plate bending problems}, Numer. Math., 21 (1973), pp.~43--62.

\bibitem{Kadlec1964}
{\sc J.~Kadlec}, {\em The regularity of the solution of the {Poisson} problem
  in a domain whose boundary is similar to that of a convex domain},
  Czechoslovak Math. J., 14 (1964), pp.~386--393.

\bibitem{LiMingZhou2025TRUNC}
{\sc H.~Li, P.~Ming, and Y.~Zhou}, {\em The {TRUNC} element in any dimension
  and application to a modified {Poisson} equation}, Numerical Methods for
  Partial Differential Equations, 41 (2025), p.~e23151.

\bibitem{LiuHuangWang2020}
{\sc K.~Liu, X.~Huang, and W.~Wang}, {\em Mixed finite element method for
  fourth-order elliptic singular perturbation problems}, J. Wenzhou Univ. (Nat.
  Sci. Ed.), 41 (2020), pp.~24--30.
\newblock (in Chinese).

\bibitem{MitreaMitreaYan2010}
{\sc D.~Mitrea, M.~Mitrea, and L.~Yan}, {\em Boundary value problems for the
  {Laplacian} in convex and semiconvex domains}, J. Funct. Anal., 258 (2010),
  pp.~2507--2585.

\bibitem{Nedelec1986}
{\sc J.-C. N\'{e}d\'{e}lec}, {\em A new family of mixed finite elements in
  {$\mathbb{R}^3$}}, Numer. Math., 50 (1986), pp.~57--81.

\bibitem{NilssenTaiWinther2001}
{\sc T.~K. Nilssen, X.-C. Tai, and R.~Winther}, {\em A robust nonconforming
  {$H^2$}-element}, Math. Comp., 70 (2001), pp.~489--505.

\bibitem{Nitsche1971}
{\sc J.~Nitsche}, {\em {\"{U}ber} ein {Variationsprinzip} zur {L\"{o}sung} von
  {Dirichlet-Problemen} bei {Verwendung} von {Teilr\"{a}umen}, die keinen
  {Randbedingungen} unterworfen sind}, Abh. Math. Sem. Univ. Hamburg, 36
  (1971), pp.~9--15.

\bibitem{PechsteinSchoeberl2011}
{\sc A.~Pechstein and J.~Sch\"oberl}, {\em Tangential-displacement and
  normal-normal-stress continuous mixed finite elements for elasticity}, Math.
  Models Methods Appl. Sci., 21 (2011), pp.~1761--1782.

\bibitem{PechsteinSchoeberl2017}
{\sc A.~S. Pechstein and J.~Sch\"oberl}, {\em The {TDNNS} method for
  {Reissner--Mindlin} plates}, Numer. Math., 137 (2017), pp.~713--740.

\bibitem{Semper1992}
{\sc B.~Semper}, {\em Conforming finite element approximations for a
  fourth-order singular perturbation problem}, SIAM J. Numer. Anal., 29 (1992),
  pp.~1043--1058.

\bibitem{Talenti1965}
{\sc G.~Talenti}, {\em Sopra una classe di equazioni ellittiche a coefficienti
  misurabili}, Ann. Mat. Pura Appl., 69 (1965), pp.~285--304.

\bibitem{WangMeng2007}
{\sc M.~Wang and X.~Meng}, {\em A robust finite element method for a 3-{D}
  elliptic singular perturbation problem}, J. Comput. Math., 25 (2007),
  pp.~631--644.

\bibitem{WangXuHu2006}
{\sc M.~Wang, J.~C. Xu, and Y.~C. Hu}, {\em Modified {Morley} element method
  for a fourth order elliptic singular perturbation problem}, J. Comput. Math.,
  24 (2006), pp.~113--120.

\bibitem{WangHuangTangZhou2018}
{\sc W.~Wang, X.~Huang, K.~Tang, and R.~Zhou}, {\em {Morley-Wang-Xu} element
  methods with penalty for a fourth order elliptic singular perturbation
  problem}, Adv. Comput. Math., 44 (2018), pp.~1041--1061.

\bibitem{ZhangZhao2024}
{\sc B.~Zhang and J.~Zhao}, {\em The virtual element method with interior
  penalty for the fourth-order singular perturbation problem}, Commun.
  Nonlinear Sci. Numer. Simul., 133 (2024), p.~Paper No. 107964.

\bibitem{ZhangZhaoChen2020}
{\sc B.~Zhang, J.~Zhao, and S.~Chen}, {\em The nonconforming virtual element
  method for fourth-order singular perturbation problem}, Adv. Comput. Math.,
  46 (2020), p.~Paper No. 19.

\bibitem{Zulehner2011}
{\sc W.~Zulehner}, {\em Nonstandard norms and robust estimates for saddle point
  problems}, SIAM J. Matrix Anal. Appl., 32 (2011), pp.~536--560.

\end{thebibliography}

\end{document}